\documentclass[11pt, a4paper]{article}
\pdfoutput=1  
\usepackage{latexsym,amssymb,amsfonts,amsmath,epsfig,tabularx}

\usepackage{pdfsync}

\usepackage[hidelinks=true]{hyperref}

\usepackage{cite}

\usepackage{verbatim}

\usepackage{enumitem} 

\usepackage[usenames]{xcolor}
\newif\ifmarkauthors
\markauthorstrue
\ifmarkauthors
  \definecolor{darkred}{RGB}{139,0,0}
  \definecolor{darkgreen}{RGB}{0,100,0}
  \definecolor{darkmagenta}{RGB}{139,0,139}
  \definecolor{darkorange}{RGB}{190,70,20}

  \def\cbdelete[#1]{}
\else

  \def\cbdelete[#1]{}
\fi

\newtheorem{theorem}{Theorem}
\newtheorem{remark}[theorem]{Remark}
\newtheorem{lemma}[theorem]{Lemma}
\newtheorem{proposition}[theorem]{Proposition}
\newtheorem{corollary}[theorem]{Corollary}
\newtheorem{example}[theorem]{Example}

\def\reals{\mathbb{R}}

\def\natural{\mathbb{N}}
\def\bbN{\mathbb{N}}

\def\il{\left<}
\def\ir{\right>}
\def\eps{\varepsilon}
\def\e{\varepsilon}

\newcommand{\ba}{{\boldsymbol{a}}}

\newcommand{\bx}{{\boldsymbol{x}}}

\newcommand{\setu}{{\mathfrak{u}}}
\newcommand{\setv}{{\mathfrak{v}}}

\newcommand{\setU}{{\mathcal{U}}}

\newenvironment{proof}{\begin{trivlist}\item[\hskip\labelsep{\bf Proof.}]}{$\hfill\Box$\end{trivlist}}

\newcommand{\ex}{{\mathrm{e}}}
\newcommand{\rd}{{\mathrm{d}}}

\newcommand{\calA}{{\mathcal{A}}}
\newcommand{\calD}{{\mathcal{D}}}
\newcommand{\calF}{{\mathcal{F}}}
\newcommand{\calI}{{\mathcal{I}}}

\newcommand{\calL}{{\mathcal{L}}}

\newcommand{\calU}{{\mathcal{U}}}
\newcommand{\bsz}{{\boldsymbol{z}}}
\newcommand{\bsi}{{\boldsymbol{i}}}

\newcommand{\bsx}{{\boldsymbol{x}}}
\newcommand{\bsy}{{\boldsymbol{y}}}

\newcommand{\bsDelta}{{\boldsymbol{\Delta}}}
\newcommand{\bszero}{{\boldsymbol{0}}}

\newcommand{\bsgamma}{{\boldsymbol{\gamma}}}
\newcommand{\bsalpha}{{\boldsymbol{\alpha}}}

\newcommand{\bbZ}{{\mathbb{Z}}}

\newcommand{\bbR}{{\mathbb{R}}}

\newcommand{\mask}[1]{}

\title{Infinite-dimensional integration\\
and the multivariate decomposition method}

\author{F. Y. Kuo, D. Nuyens, L. Plaskota, I. H. Sloan, and G. W. Wasilkowski}

\date{9 September 2016}

\begin{document}

\maketitle

\begin{abstract}
We further develop the \emph{Multivariate Decomposition Method} (MDM) for
the Lebesgue integration of functions of infinitely many variables $x_1,x_2,x_3,\ldots$
with respect to a corresponding product of a one dimensional probability measure.
The method is designed for functions that admit a dominantly convergent decomposition
$f=\sum_\setu f_\setu$, where $\setu$ runs over all finite subsets of positive integers, and
for each $\setu=\{i_1,\ldots,i_k\}$ the function $f_\setu$ depends only on
$x_{i_1},\ldots,x_{i_k}$.

Although a number of concepts of
infinite-dimensional integrals have been used in the literature, questions
of uniqueness and compatibility have mostly not been studied.
We show that, under appropriate convergence conditions, the Lebesgue integral
equals the `anchored' integral, independently of the anchor.

For approximating the integral, the MDM assumes that point
values of $f_\setu$ are available for important subsets $\setu$, at some
known cost. In this paper we introduce a new setting, in which it is
assumed that each $f_\setu$ belongs to a normed space $F_\setu$, and that
bounds $B_\setu$ on $\|f_\setu\|_{F_\setu}$ are known. This contrasts with
the assumption in many papers that weights $\gamma_\setu$, appearing in
the norm of the infinite-dimensional function space, are somehow known.
Often such weights $\gamma_\setu$ were determined by minimizing an error
bound depending on the $B_\setu$, the $\gamma_\setu$ \emph{and} the chosen
algorithm, resulting in weights that depend on the algorithm. In contrast,
in this paper only the bounds $B_\setu$ are assumed known. We give two
examples in which we specialize the MDM: in the first case $F_\setu$ is
the $|\setu|$-fold tensor product of an anchored reproducing kernel
Hilbert space, and in the second case it is a particular non-Hilbert space
for integration over an unbounded domain.
\end{abstract}

\section{Introduction} \label{sec:intro}

This paper is intended as a contribution, both theoretical and practical,
to the challenging task of numerical integration of multivariate
functions, when the number of variables is large, and even infinite.

High-dimensional integration has emerged in recent years as a significant
new direction for numerical computation, see \cite{DST16}. On the one hand
the improved processing power of computers has encouraged the practical
computation of multivariate integrals with very large numbers of
variables, into the hundreds or thousands or tens of thousands. On the
other hand such problems will never become trivial -- indeed, many
important high-dimensional problems (see below for an example) contain
parameters for which physically interesting choices can lead to problems
of unlimited difficulty.

Many papers have been written in recent decades about high-dimensional
integration, see the reviews \cite{BG04} for sparse grid methods and
\cite{DKS13} for Quasi-Monte Carlo methods, and their many references.

In many applications the number of variables is not merely large
but in principle infinite.  This is the case in the important class of
problems of elliptic partial differential equations with uncertain
coefficients. A key example is Darcy flow through a porous medium with
highly variable permeability (see e.g., \cite{GKNSS11}), with the
permeability modelled as a random field. Since a continuous random field
requires an infinite number of scalar random variables for its
description, the expected value of any property of the flow is in
principle an infinite-dimensional integral.  Of course in practice the
infinite-dimensional integral has to be truncated to a finite-dimensional
integral, but if the correlation length of the permeability field is
small, or the variance is large, then the dimensionality might need to be
very large indeed to capture the essential physics.

Some other recent papers devoted to infinite-dimensional integration are
\cite{Bal12,BalGne14,CDMK09,DicGne14a,DicGne14b,Gne10,Gne12,Gne13,GMR14,
HMNR10,HW01,KSWW10a,NiuHic10,NHMR10,PW11,PW14,PWW00,
Was12a,Was12b,Was13a,Was14a,Was14b,WW96,WW11a,WW11b}.

The theoretical setting in most of these papers (and the theoretical
setting is of key importance, given the general lack of useful intuition
in high dimensions) has been that of ``weighted'' reproducing kernel
Hilbert spaces, where the weights enter the norm of the function space.
Often the spaces are tensor products of one-dimensional reproducing kernel
Hilbert spaces, and the ``weights'' are of the ``product'' kind introduced
by Sloan and and Wo\'zniakowski \cite{SW98}, in which there is one weight
$\gamma_j$ for each variable $x_j$, with $\gamma_1 \ge \gamma_2 \ge \cdots
> 0$, with the decreasing weights reflecting the decreasing importance of
the successive variables. This has been extended to ``general weights'',
in which case there is a potentially different non-negative weight
$\gamma_\setu$ for each finite subset $\setu$ of the natural numbers; the
product weight case  is then recovered by setting
$\gamma_\setu=\prod_{j\in\setu}\gamma_j$. In these papers it is assumed
that the integrand is expressible in the form
\begin{equation} \label{eq:decomp}
 f(\bsx) \,=\, \sum_{|\setu|<\infty}f_\setu(\bsx)
 \,=\, \sum_{|\setu|<\infty}f_\setu(\bsx_\setu),
\end{equation}
where the sum (see also \eqref{eq:sumlim} below) is over all finite
subsets $\setu\subset\bbN:=\{1,2,\ldots\}$, with $|\setu|$ denoting the
cardinality of $\setu$, and each function $f_\setu$ depends only on the
subset of the variables in $\setu$ and we write $\bsx_\setu:=\{x_j:j\in
\setu\}$. Furthermore, $f_\setu$ is assumed to belong to a normed space
$F_\setu$, which is usually a reproducing kernel Hilbert space. The
weights determine the importance of different subsets $\setu$ through
their appearance in a norm of the form
\begin{equation}\label{eq:papnorm}
 \Bigg(\sum_{|\setu|<\infty}
 \bigg(\frac{\|f_\setu\|_{F_\setu}}{\gamma_\setu}\bigg)^{q}\Bigg)^{1/q}
\end{equation}
in which case the function $f$ belongs to a Banach space
$\calF=\calF_{\bsgamma,q}$.  Commonly $q=2$, in which case $\calF$ is a
Hilbert space, and some papers replace $\gamma_\setu$ by
$\gamma_\setu^{1/2}$.

Throughout this paper we shall use the convention that in infinite sums
over subsets as in \eqref{eq:decomp} and~\eqref{eq:papnorm} the terms are
to be ordered in terms of increasing truncation dimension,
\begin{equation} \label{eq:sumlim}
 \sum_{|\setu|<\infty}f_\setu(\bsx)
 \,:=\, \lim_{d\to\infty}\sum_{\setu\subseteq\{1,\ldots,d\}}f_\setu(\bsx),
\end{equation}
with the additional subsets when $d$ increases to $d+1$ ordered as for the
case $d$ with respect to the original members.

The \emph{Multivariate Decomposition Method} (MDM) proposed in
\cite{Was13a,Was14a,PW14} is a generalisation of the \emph{Changing
Dimension Algorithm} in \cite{KSWW10a,PW11}. Here too it is assumed that
an expansion of the form \eqref{eq:decomp} exists, and that values of
$f_\setu(\bsx_\setu)$, while not available explicitly, can be obtained by
a modest number of evaluations of $f(\bsx)$.  In Section~\ref{sec:applic1}
we give specific examples in which values of $f_\setu(\bsx_\setu)$ can be
obtained from at most $2^{|\setu|}$ evaluations of $f(\bsx)$ -- an
acceptably small number if the cardinality $|\setu|$ is small.

In previous MDM papers \cite{Was13a,Was14a,PW14} the norm was assumed to
be of the form \eqref{eq:papnorm}. In the present paper, in contrast, we
do not assume that weights $\gamma_\setu$ are given as \emph{a priori}
information.  Rather, we assume that the terms $f_\setu$ in the expansion
\eqref{eq:decomp} belong to normed function spaces $F_\setu$, and
crucially, that upper bounds on the norms $\|f_\setu\|_{F_\setu}$ are
known, i.e., that for $|\setu|<\infty$ numbers $B_\setu$ satisfying
\begin{align} \label{eq:Bu}
  \|f_\setu\|_{F_\setu} \,\le\, B_\setu
\end{align}
are given as \emph{a priori} information. This new setting is equivalent
to putting $q=\infty$ and $\gamma_\setu = B_\setu$ in \eqref{eq:papnorm},
but in all other cases the two approaches are not equivalent. Given the
bounds $B_\setu$, one is of course free to select a value of $q$ and a set
of weights $\gamma_\setu$ that make \eqref{eq:papnorm} finite, but in
neither case is the choice unique, or is one choice obviously better than
another. In the present work we make no explicit use of weights.

The advantage of the current setting, in which bounds $B_\setu$ rather
than weights $\gamma_\setu$ are specified, lies in its immediate
applicability once such bounds are known. We note that a number of recent
papers have provided directly useable bounds $B_\setu$: for partial
differential equations with random coefficients, see \cite{CDS10} for the
case of uniformly distributed stochastic variables, and
\cite{GKNSSS15,HPS} for the lognormal case; and for generalized response
models in statistics, see \cite{SKS13}. It seems likely that similar
bounds will be found for other applications in the future. In contrast, in
practical situations it is typically not clear how to choose weights
$\gamma_\setu$, or product weights $\gamma_j$. We recall that weights were
originally introduced (in \cite{SW98}) to provide a setting in which the
tractability of multivariate integration (roughly, to know what happens to
the worst-case error as $d\to\infty$) could be studied. Never was it
claimed that weights were naturally available in an application. Some
recent papers have obtained formulas for suitable weights, but these
``optimal weights'' are deduced by minimising an error bound which depends
on the $B_\setu$, the $\gamma_\setu$ \emph{and} the chosen algorithm. See
\cite{KSS12,KSS15} and \cite{GKNSSS15,KSSSU} for the PDE with random
coefficients and randomly shifted lattice rules in the uniform and
lognormal cases respectively, and \cite{DKLNS14,DKLS} for the uniform case
with higher order digital nets, as well as \cite{KN} for a survey of these
results. The dependence of weights on the algorithm is unacceptable from
the point of view of information based complexity, where the complexity of
the problem (i.e., integration in $\calF_{\bsgamma,q}$, specified
by~\eqref{eq:papnorm}, and thus depending on the $\gamma_\setu$) is
supposed to be studied independently of any algorithm. In contrast, we
believe that the present setting will provide a robust basis not only for
the development of computational schemes, but also for future complexity
and tractability studies of high-dimensional integration.

\smallskip
The problem to be considered is that of integration
over an infinite-dimensional product region,
\begin{equation}\label{eq:Lebesgue}
  \calI(f)\,:=\,\int_{D^\bbN} f(\bsx)\,\rd\mu(\bsx),
\end{equation}
where the integral is in the Lebesgue sense.
Here $\mu$ is the countable product, $\mu=\times_{j=1}^\infty\mu_1$, of
a one-dimensional probability measure $\mu_1$ defined on a Borel set
$D\subseteq\bbR$, and $D^\bbN$ is the set of all infinite sequences
$\bsx=(x_1,x_2,x_3,\ldots)$ with $x_j\in D$. We assume that $\mu_1$
is determined by a probability density $\rho$ on $D$.

At this point it is worth mentioning that many papers define
the infinite-dimensional integral as
\begin{equation}\label{eq:diffdef}
   \calI_a(f) \,:=\, \lim_{d\to\infty}\int_{D^d}f(x_1,\ldots,x_d,a,a,\ldots)\,
    \rd\mu_d (x_1,\ldots,x_d),
\end{equation}
where $\mu_d=\times_{j=1}^d\mu_1$,
for some fixed $a\in D$, usually with the `anchor' taken as $a=0$.
Taken on its own that definition seems open to question if the value of
the integral could depend on the choice of the anchor $a$. We address
this uniqueness concern in Section \ref{sec:definingint}, where sufficient 
conditions are given to ensure that the equality
$\calI(f)=\calI_a(f)$ holds independently of the choice of $a$.

\smallskip
To approximate the integral \eqref{eq:Lebesgue}, the MDM uses the
decomposition of $f$ given in \eqref{eq:decomp}. Indeed, assuming that the
partial sums in \eqref{eq:sumlim} converge dominantly to $f$, we have
\begin{align} \label{eq:int}
  \calI(f)
 \,=\, \lim_{d\to\infty}\sum_{\setu\subseteq\{1,\dots,d\}}
       \int_{D^{|\setu|}} f_\setu(\bsx_\setu)\,\rd\mu_{|\setu|}(\bsx_\setu)
 \,=:\,\lim_{d\to\infty}\sum_{\setu\subseteq\{1,\dots,d\}} I_\setu (f_\setu)
 \,=\, \sum_{|\setu|<\infty} I_\setu (f_\setu).
\end{align}
The essence of MDM is that a separate quadrature rule (which in all but a
finite number of cases will be the zero approximation) is applied to each
term $f_\setu$ in the decomposition of~$f$. In more detail, the overall
algorithm $\calA_\e$ for approximating $\calI(f)$ up to an error
request~$\eps$ has the form
\begin{equation} \label{eq:alg}
  \calA_\e(f) \,:=\, \sum_{\setu\in \calU(\e)} A_{\setu,n_\setu}(f_\setu)\,:=\,
   \sum_{\setu\in \calU(\e)}\sum_{i=1}^{n_\setu} w_{\setu,i}\, f_\setu(\bsx_{\setu, i}),
\end{equation}
where the \emph{active set} $\calU(\e)$ is a finite set of finite subsets
of $\bbN$, and $n_\setu$, $\bsx_{\setu,i}$ and $w_{\setu,i}$ are
parameters of the quadrature rule $A_{\setu,n_\setu}$.  In effect, the
contributions to $\calI(f)$ from terms $f_\setu$ with $\setu$ outside the
active set are approximated by zero and thus the construction of
$\calU(\e)$ depends on the error request~$\e$.  We note that both the
active set $\calU(\eps)$ and the algorithm $\calA_\eps$ are intended to be
independent of the particular function $f\in\calF$ once the bounds
$B_\setu$ are provided.

Clearly, the selection of the active set $\calU(\e)$ and the determination
of the quadrature rules $A_{\setu,n_\setu}$ for $\setu\in\calU(\e)$ are
key ingredients of the MDM.  To make the selections in a rational way we
need to assume not only that we have \emph{a priori} information about the
size of the terms $f_\setu$ in the expansion of~$f$ in the form of the
upper bounds $B_\setu$, but also that we are provided with suitable
information about the difficulty of the integration problem and the
quality of the quadrature rules in $F_\setu$.

For two specific applications, we develop an MDM whose worst-case error is
upper bounded by $\e^{1-\delta(\e)}$ and the information cost is
proportional to $(1/\e)^{1+\delta(\e)}$ where $\delta(\e)>0$ and
$\delta(\e)\to0$, under quite general assumptions about
the bounds $B_\setu$ and the cost of
function evaluations. This means that for the given application the MDM is
almost optimal since even for the corresponding space of univariate
functions, the minimal cost of computing an $\e$-approximation is
proportional to $1/\e$ for these two applications.

The content of the paper is as follows. In Section~\ref{sec:definingint} we
discuss relations between the Lebesgue and `anchored' integrals.
Then in Section~\ref{sec:setting} we describe the setting for the MDM and in
Section~\ref{sec:mdm} we develop the MDM in its general form. Then in
Section~\ref{sec:applic1} we turn to an important application, one that
provides the initial motivation for the method. This is the case of the
so-called ``anchored decomposition'' associated with anchored reproducing
kernel Hilbert spaces, which have very often been used in studies of
multivariate integration and approximation. We shall see that in this case
all the assumptions of the MDM are satisfied, and that there is a
significant class of integration problems for which the MDM can be highly
efficient. In Section~\ref{sec:SecondAppl} we consider another
application, this one of a non-Hilbert space nature. Efficient
implementation of MDM will be considered in a forthcoming
paper\cite{GKNW}.

\section{Lebesgue integral and `anchored' integral}\label{sec:definingint}

In this section, we compare the Lebesgue integral \eqref{eq:Lebesgue} with
the `anchored' integral \eqref{eq:diffdef} and show in particular that, under suitable
assumptions, they are equivalent.

Recall that $D$ is a Borel subset of $\bbR$ and $D^\bbN$, where
$\bbN=\{1,2,3,\ldots\}$, is the set of all infinite sequences/points
$\bsx =(x_1,x_2,x_3,\dots)$ with each $x_j\in D$. Furthermore, $\mu_1$ is
a probability measure on the Borel $\sigma$-field of $D$, and we denote by
$\mu_d=\times_{j=1}^d\mu_1$ the $d$-product of $\mu_1$ on $D^d$, and by
$\mu=\times_{j=1}^\infty\mu_1$ the countable product of $\mu_1$ on $D^\bbN$.
By Lebesgue integral of a function $f:D^\bbN\to\bbR$ we mean the integral with
respect to $\mu$, and $a.e.$ means 
\emph{almost everywhere with respect to $\mu$.}

In general, the Lebesgue and `anchored' integrals are quite different.
Moreover, `anchored' integrals may depend on $a$. A simple
example is provided by the function $f:\bbR^\bbN\to\bbR$ such that
$f(\bsx)=0$ if $\bsx$ has only finitely many non-zero coordinates, and
$f(\bsx)=1$ otherwise. Indeed, then the limit in \eqref{eq:diffdef} is $0$
for $a=0,$ and $1$ for $a=1$. On the other hand, for the integral
\eqref{eq:Lebesgue} we have $\calI(f)=1$.

\smallskip
The following well-known {\em Lebesgue's dominated convergence
theorem}, see, e.g., \cite[Sect.\,26]{Halmos74}, will play an important
role in our considerations.
\begin{theorem}\label{thm:DCT}
Let $\{f_d\}_{d\ge 0}$ be a sequence of integrable functions that
converges dominantly to $f$, i.e.,
\begin{itemize}
\item[\textnormal{(i)}] $\lim_{d\to\infty}f_d(\bsx)=f(\bsx)\;
                           \mbox{ for }\;\bsx\;a.e.$
\item[\textnormal{(ii)}] for some integrable function $g$ we have
                   $|f_d(\bsx)|\le |g(\bsx)|\;
                        \mbox{ for }\; \bsx\;a.e.$
\end{itemize}
Then
\begin{equation}\label{efdilimit}
    \calI(f)\,=\,\lim_{d\to\infty}\calI(f_d).
\end{equation}
\end{theorem}

Theorem \ref{thm:DCT} can be directly applied to a variety of
sequences $\{f_d\}_{d\ge 0}$. It implies, in particular, that if the
functions $f_\setu$ in the decomposition \eqref{eq:decomp} are integrable
and
\begin{equation}\label{dec}
    f_d\,:=\,\sum_{\setu\subseteq\{1,\ldots,d\}}f_\setu
\end{equation}
converge dominantly to $f$, then the integral can be computed using the equality
$$   \mathcal I(f)\,=\,\lim_{d\to\infty}\sum_{\setu\subseteq\{1,\ldots,d\}}
       \int_{D^{|\setu|}}f_\setu(\bx_\setu)\,\rd\mu_{|\setu|}(\bx_\setu).
$$
(We will use this fact later in the MDM.)
Similarly, for fixed  $\ba=(a_1,a_2,a_3,\ldots)\in D^\bbN$ we have
\begin{equation}\label{dom-conv}
      \calI(f)\,=\,
      \calI_\ba(f)\,:=\,\lim_{d\to\infty}\int_{D^d}
      f(x_1,\ldots,x_d,a_{d+1},a_{d+2},\ldots)\,\rd\mu_d(x_1,\ldots,x_d)
\end{equation}
if the functions
\begin{equation}\label{ef-d}
   D^\bbN\,\ni\,\bx\,=\,
    (x_1,x_2,x_3,\ldots)\,\mapsto\,
    f(x_1,\ldots,x_d,a_{d+1},a_{d+2},\ldots)
\end{equation}
are integrable and converge dominantly to $f$. 
That is, we then have equivalence of the Lebesgue and
`anchored' integrals for the anchor $\ba$. 

The following result allows us to claim that such equivalence holds for \emph{all}
anchors after checking only one sequence, e.g.,
$f_d(\bsx)=f(x_1,\ldots,x_d,a,a,a,\ldots)$ for an $a\in D$.
The sufficient condition is however stronger than that in Theorem \ref{thm:DCT}.

\begin{theorem}\label{thm:any-anchor}
Let $\{f_d\}_{d\ge 0}$ be a sequence of integrable functions converging
dominantly to $f$, such that each $f_d$ depends only on the variables
$x_1,\ldots,x_d$.  Let, in addition, the convergence be uniform $a.e.$, or
the following less restrictive assumption be satisfied: for 
$\bsx=(x_1,x_2,x_3,\ldots)\,\;a.e.$
\begin{equation}\label{uniform}
   |f_d(\bsx)-f(\bsx)|\,\le\, g_d(x_1,\ldots,x_d)\qquad\forall d\ge 0,
\end{equation}
for some sequence $\{g_d\}_{d\ge 0}$ of integrable functions that converges
dominantly to the zero function. Then \eqref{dom-conv} holds for 
$\ba=(a_1,a_2,a_3,\ldots)\;\,a.e.$ 

Moreover, if \eqref{uniform} holds for all $\bx\in D^\bbN$ then \eqref{dom-conv} 
holds for all anchors $\ba\in D^\bbN$ provided the functions \eqref{ef-d} are 
measurable.
\end{theorem}

\begin{proof}
By dominated convergence of $\{f_d\}_{d\ge 0}$ we know that $f$ is integrable.

Let $\calD_1$ be the set of all points $\ba=(a_1,a_2,a_3,\ldots)$ for which
the functions \eqref{ef-d}
are measurable for all $d\ge 0$. We show that $\mu(\calD_1)=1$.  Indeed, since
$f$ is measurable and $\mu$ is a product measure, $\mu=\mu_d\times\mu$, it
follows that for any fixed $d$ the measure of
$(a_{d+1},a_{d+2},\ldots)$ such that the functions \eqref{ef-d} are
measurable is $1$.
Hence $\mu(\calD_1)=1$ as an intersection of countably many sets of measure $1$.

Let $\calD_2$ consist of all $\ba$ for which
\begin{equation}\label{D2set}
      \mu\Big(\big\{\,\bsx:\;
      |f_d(\bsx)-f(x_1,\ldots,x_d,a_{d+1},a_{d+2},\ldots)|\,\le\,g_d(x_1,\ldots,x_d)\;\,
      \forall d\ge 0\,\big\}\Big) \,=\,1.
\end{equation}
Then $\mu(\calD_2)=1$ as well. Indeed, denote by $\calD$ the
collection of all $\bsx$ for which \eqref{uniform} holds. Since
$\mu(\calD)=1$, using again the argument that $\mu$ is a product measure,
we obtain, for 
$\ba\,\;a.e.$ and fixed $d\ge 0$, that
$$  \mu\Big(\big\{\bsx:(x_1,\ldots,x_d,a_{d+1},a_{d+2},\ldots)\in\calD\big\}\Big)=
       \mu_d\Big(\big\{(x_1,\ldots,x_d):(x_1,\ldots,x_d,a_{d+1},a_{d+2},\ldots)\in\calD\big\}\Big)
       =1.
$$
Since $\calD_2$ is the countable intersection of these sets, it
follows that $\mu(\calD_2)=1$.
Concluding this part of the proof we have that
$\mu(\calD_1\cap\calD_2)=1$.

Now let $\ba\in\calD_1\cap\calD_2$. Then for 
$\bsx\,\;a.e.$
\begin{eqnarray*}
      |f(\bsx)-f(x_1,\ldots,x_d,a_{d+1},a_{d+2},\ldots)| &\le &|f_d(\bsx)-f(\bsx)|+
       |f_d(\bsx)-f(x_1,\ldots,x_d,a_{d+1},a_{d+2},\ldots)| \\
        &\le & 2\,g_d(x_1,\ldots,x_d),
\end{eqnarray*}
and by dominated convergence of $\{g_d\}_{d\ge 0}$ to the zero
function, it follows that
$$ \lim_{d\to\infty}\int_{D^\bbN}
      |f(\bsx)-f(x_1,\ldots,x_d,a_{d+1},a_{d+2},\ldots)|\,\rd\mu(\bsx)\,\le\,2
      \lim_{d\to\infty}\int_{D^d}g_d(x_1,\ldots,x_d)\,\rd\mu_d(x_1,\ldots,x_d)=0,
$$
which implies \eqref{dom-conv} for 
$\ba\,\;a.e.$

To show the remaining part of the theorem, observe that the actual set of anchors
$\ba$ for which we have equvalence of the Lebesgue and `anchored' integrals
includes $\mathcal D_1\cap\mathcal D_2$. Under the additional assumptions we
obviously have $\mathcal D_1=D^\bbN$ and $\mathcal D_2=D^\bbN$.
Hence $\mathcal D_1\cap\mathcal D_2=D^\bbN$, as claimed.
\end{proof}

\begin{example} \label{eq:calD}
For $D=[-\alpha,\alpha]$ or $D=\bbR$, consider
\begin{equation}\label{expl1}
    f(x_1,x_2,\ldots) := \sum_{j=1}^\infty\lambda_j\,x_j^2,
\end{equation}
where $\lambda_j>0$ $\forall j$ and $\sum_{j=1}^\infty \lambda_j<\infty$,
and assume that $E:=\int_D x^2\,\rd\mu_1< \infty.$ The function $f$ is well defined
since even in case $D=\bbR$ (where $\mu_1$ could be Gaussian) the set
of sequences $(x_1,x_2,\ldots)\in\bbR^\bbN$ for which the sum \eqref{expl1} is 
finite is of measure one. The Lebesgue integral equals
$$  \calI(f)\,=\,\sum_{j=1}^\infty \int_D\lambda_j x_j^2\,\rd\mu_1\,=\,
     E\sum_{j=1}^\infty\lambda_j.
$$
On the other hand, for the `anchored' integral with $\ba=(a_1,a_2,a_3,\ldots)$ 
we have
$$ \int_{D^d}f(x_1,\ldots,x_d,a_{d+1},a_{d+2},\ldots)
   \,\rd\mu_d(x_1,\ldots,x_d) \,=\,
   E\sum_{j=1}^d\lambda_j\,+\, \sum_{j=d+1}^\infty\lambda_j a_j^2,
$$
which converges to $\calI(f)$ if and only if 
\begin{equation}\label{eq:anccond}
   \sum_{j=1}^\infty\lambda_ja_j^2<\infty.
\end{equation}
The inequality \eqref{eq:anccond} is a necessary and sufficient condition for 
$\calI(f)=\calI_\ba(f)$. It holds for all $\ba$ if $D=[-\alpha,\alpha]$,
and for $\ba\,\;a.e.$ if $D=\bbR$.
We also observe that assumptions of
Theorem \ref{thm:any-anchor} are satisfied with, e.g.,
$$ f_d(\bsx)\,=\,f(x_1,\ldots,x_d,0,0,0,\ldots)
      \,=\,\sum_{j=1}^d\lambda_jx_j^2\qquad\mbox{and}\qquad
      g_d\,\equiv\,\alpha^2\sum_{j={d+1}}^\infty\lambda_j
$$
only if $D$ is a finite interval, and then the convergence is uniform 
on the whole domain $D^\bbN$.
\end{example}

\section{The setting for MDM}\label{sec:setting}

We provide in this section basic assumptions concerning the approximate
integration problem considered in this paper. In particular, we will
introduce standing assumptions \ref{A1}--\ref{A6} that pertain
to the whole paper. As we shall see later, the assumptions are satisfied
by a number of specific problems.

\subsection{The function class $\calF$}

We introduce a class $\calF$ of $\infty$-variate real-valued functions
whose integrals are to be approximated. For a finite subset
$\setu\subset\bbN$ and a point $\bsx\in D^\bbN$,
$\bsx_\setu$ denotes the variables
$x_j$ with $j\in \setu$, and $D^\setu$ denotes the product integration
region $D^{|\setu|} $ with the variables replaced by those in
$\bsx_\setu$, where $|\setu|$ denotes the cardinality of $\setu$.

As in the Introduction, functions in the class $\calF$ are expressed as
sums of functions $f_\setu$, with $f_\setu$ depending only on the
variables in $\bsx_\setu$ and belonging to a normed linear space
$F_\setu$.
\begin{enumerate}
[start=1,label=\textup{(\textbf{A\arabic*})},leftmargin=2.5em,series=assumptions]
\item\label{A1} %
Each $f\in\calF$ has a decomposition of the form \eqref{eq:decomp}
where the sums over all finite subsets $\setu\subset\bbN$ are defined
as in \eqref{eq:sumlim}, and each $f_\setu$ is formally a function on
$D^\mathbb{N}$ but depends only on the variables $x_j$ with
$j\in\setu$. The functions $f_d$ defined in \eqref{dec} are
assumed to be dominantly (or even uniformly) convergent to $f$.
\end{enumerate}
\begin{enumerate}[resume*=assumptions]
\item\label{A2} Each component $f_\setu$ of $f$ in
    \eqref{eq:decomp} belongs to a normed space $F_\setu$ of
    real-valued measurable functions defined on $D^\setu$ with norm
    $\|f_\setu\|_{F_\setu}$. Moreover, $\|f_\setu\|_{F_\setu} \le
    B_\setu$, see \eqref{eq:Bu}, for known positive numbers~$B_\setu$.
    In particular, $F_\emptyset$ is the space of constant functions
    with norm given by the absolute value. Finally, point evaluation
    at any $\bsx_\setu\in D^\setu$ is assumed to be a continuous
    linear functional on $F_\setu$.
\end{enumerate}

We shall see later in Sections~\ref{sec:applic1} and~\ref{sec:SecondAppl}
concrete examples of the function class $\calF$.

\subsection{The integration problem}

As in \eqref{eq:int}, we express an infinite-dimensional integral
$\calI(f)$ for $f\in\calF$ as a sum of multivariate integrals
$I_\setu(f_\setu)$ from the decomposition $f=\sum_{|\setu|<\infty}
f_\setu$. Recalling that $\rho$ is a given probability density function
on $D$, we write $\rho_\setu(\bsx_\setu) := \prod_{j\in\setu}\rho(x_j)$.
For $\setu=\emptyset$, we set $I_\emptyset(f_\emptyset)\,:=\,f_\emptyset$.
We make the following assumption.
\begin{enumerate}[resume*=assumptions]
\item\label{A3}
All functions in $F_\setu$ are Lebesgue measurable and integrable with
respect to $\rho_{\setu}(\bx_\setu)\,\rd\bx_\setu$, and the
functionals $I_\setu$ are continuous, i.e.,
\begin{equation} \label{eq:Cu}
  C_{\setu}\,:=\,\|I_\setu\|\,=\,
  \sup_{\|f_\setu\|_{F_\setu}\le 1} \left|\int_{D^{|\setu|}} f_\setu(\bsx_\setu)\,
    \rho_\setu(\bsx_\setu)\,\rd\bsx_\setu \right| \,<\, \infty.
\end{equation}
\end{enumerate}
At this moment, we also assume that the numbers $B_\setu$ in~\eqref{eq:Bu}
and $C_{\setu}$ in~\eqref{eq:Cu} satisfy
\begin{equation}\label{ass:pre5}
   \sum_{|\setu|<\infty} C_{\setu}\,B_{\setu}<\infty,
\end{equation}
which will later be replaced by a stronger assumption~\ref{A4}. Since
$|I_\setu(f_\setu)| \le \|I_\setu\|\,\|f_\setu\|_{F_\setu} \le C_{\setu}\,
B_\setu$, the condition~\eqref{ass:pre5} implies that the sum
$\sum_{|\setu|<\infty} |I_\setu(f_\setu)|$ is finite. Moreover,
\[
  \sup_{f\in\calF}|\calI(f)| \,\le\,
  \sum_{|\setu|<\infty} C_{\setu}\,B_{\setu} \,<\, \infty.
\]

We end this subsection with the following remark.

\begin{remark}\label{rem:lx} Instead of assuming
convergence in \ref{A1}, we can impose
conditions on the point evaluation functionals as follows. Let
$L_{\setu,\bsx_\setu}$ be the point evaluation functional on $F_\setu$,
$L_{\setu,\bsx_\setu}(f_\setu)=f_\setu(\bsx_\setu)$. Assume that for every
$\setu$, we have
$\calL_\setu:=\sup_{\bsx_\setu}\|L_{\setu,\bsx_\setu}\|<\infty$ and
\begin{equation}\label{ass:lx2}
\sum_{|\setu|<\infty} \calL_\setu\, B_\setu <\infty.
\end{equation}
Then $|\sum_{|\setu|<\infty}f_\setu(\bsx_\setu)|\le \sum_{|\setu|<\infty}
|L_{\setu,\bsx_\setu}(f_\setu)| \le \sum_{|\setu|<\infty} \calL_\setu
\|f_\setu\|_{F_\setu} \le \sum_{|\setu|<\infty} \calL_\setu B_\setu<\infty$;
that is, we have uniform convergence on the whole domain $D^\bbN$.
Hence, defining the class
$$   \calF^*\,:=\,\bigg\{\sum_{|\setu|<\infty}f_\setu\,:
         \,\|f_\setu\|_{F_\setu}\le B_\setu\bigg\},
$$
we have by Theorem~\ref{thm:any-anchor} that for any $f$ in $\calF^*$
the Legesgue integral equals `anchored' integral for any anchor, and
the integral can be expressed by any decomposition.

For spaces $F_\setu$ being the $|\setu|$-fold tensor products of
some space $F_{\{1\}}$ of univariate functions, we often have
that $C_\setu=C_{\{1\}}^{|\setu|}$ and
$\calL_\setu=\calL_{\{1\}}^{|\setu|}$. Then for many families of bounds
$B_\setu$, including bounds of the form \eqref{eq:POD} to be discussed
later, it is known that \textnormal{(}see Lemma~\ref{lem:u}
below\textnormal{)} \eqref{ass:pre5} is equivalent to \eqref{ass:lx2}
which in turn is equivalent to
\[\sum_{|\setu|<\infty}B_\setu<\infty.
\]
\end{remark}

\subsection{Examples of decompositions}\label{ssec:example}

In practice one can expect to be given the $\infty$-variate function $f$,
the integration domain $D^\bbN$ and the weight function $\rho$, after which
it is the user's task to define a suitable sequence of normed spaces $F_\setu$
and a method of decomposing $f$ into components $f_\setu \in F_\setu$.

\begin{example}\label{exa:one}
The following $\infty$-variate problem
is a variant of a simpler model problem introduced in
\cite[Section~1.5]{KSS11}, which in turn is modeled on a study
\cite{KSS12} of a diffusion problem for the flow of a liquid through a
porous medium treated as a random permeability field. In this example we
have
\begin{equation} \label{eq:example}
  D \,:=\, \big[-\tfrac{1}{2},\tfrac{1}{2}\big],
  \quad\rho(x) := 1,\quad
  f(\bsx)\, :=\, \frac{1}{1+\sum_{j=1}^\infty x_j/j^2}.
\end{equation}

For the space $F_\setu$ we here choose for simplicity a Hilbert space.
Given the requirement that point evaluation be a continuous linear
functional, we choose the simplest Hilbert space available to us, namely
\[
 F_\setu \,:=\, H^{1,...,1} \,:=\, \bigotimes_{j\in \setu} H^1,
\]
with the conventional Sobolev-type norm
\[
  \|g\|_{F_\setu} :=
  \Bigg(\sum_{\bsalpha\le(1,\ldots,1)}\|\calD^\bsalpha g\|^2_{L_2(D^\setu)}\Bigg)^{1/2},
\]
where the sum is over all multi-indices $\bsalpha$ with components
$\alpha_j\in\{0,1\}$ for $j\in\setu$, and $\calD^\bsalpha$ denotes the
appropriate weak mixed derivative. 

This choice of the normed spaces $F_\setu$ allows many different
decompositions. To illustrate this point, let us first confine our
attention to so-called ``anchored'' decompositions (see e.g.,
\cite{KSWW10b} and later), with anchor at some fixed $a \in D$. That
is, for $f\in \calF$ the terms $f_\setu\in F_\setu$ in the decomposition
\eqref{eq:decomp} are defined by the property that
\begin{equation} \label{eq:kill}
 f_\setu(\bsx_\setu) = 0 \quad \mbox{if}\quad x_j=a \quad\mbox{and}\quad j\in\setu.
\end{equation}
For each fixed value of $a\in D$, the decomposition \eqref{eq:decomp}
satisfying this property is uniquely determined. (For a given finite
subset $\setv\subset\bbN$, set $x_j=a$ for all $j\notin \setv$ in
\eqref{eq:decomp}. The only surviving terms on the right-hand side are
those $f_\setu$ for which $\setu\subseteq\setv$. Working from the smallest
subsets upwards, one proves inductively that each $f_\setu$ is uniquely
determined.) For this anchored decomposition we have
\[
 f_d(x_1,\ldots,x_d)
 \,:=\, \sum_{\setu\subseteq\{1,\ldots,d\}}f_\setu(\bsx_\setu)
 \,=\, f(x_1,\ldots,x_d,a,a,\ldots),
\]
which follows on setting $x_j = a$ for $j>d$ in the decomposition
\eqref{eq:decomp} and using the property \eqref{eq:kill}.

Since there are an infinite number of choices for~$a$, there are
correspondingly an infinite number of decompositions, which are easily
seen to be different. Under the conditions of Theorem~\ref{thm:any-anchor}
we know that each anchored decomposition will give the same value for the
exact integral $\calI(f)$, no matter the choice of the anchor $a$.
However, the MDM developed below will in general give different
approximate results for different anchors.

This choice of spaces $F_\setu$ also allows the so-called ``ANOVA''
decomposition (see e.g., \cite{KSWW10b}) and many other possibilities. For
example, a decomposition could be determined as the result of some numerical
computation.
\end{example}

\begin{example}\label{exa:two}
With the same definition of $F_\setu$ as in Example~\ref{exa:one}, we can
express the zero function by the zero decomposition, where the uniform
convergence holds true, and we obviously have $\calI(0)=0$. We now present
an example of a decomposition of the zero function which at face value
gives a non-zero value for the integral.

For $k\ge 1$, let $\psi_k(t)=2^{k+1}\big(1-\big|2^{k+1}t-1\big|\big)_+$
with $t\in D=[-\frac 12,\frac 12]$, i.e., $\psi_k(t)$ is the `hat'
function supported on $\big[0,2^{-k}\big]$ and satisfying
$\int_D\psi_k(t)\,\mathrm dt=1$. For $\bx=(x_1,x_2,x_3,\ldots)\in D^\bbN$,
we choose
\[
  f_{\{1\}}(\bx)=\psi_1(x_1), \quad\mbox{and}\quad
  f_{\{1,\ldots,k\}}(\bx)=\psi_k(x_1)-\psi_{k-1}(x_1) \quad\mbox{for}\quad k \ge 2,
\]
and set $f_\setu = 0$ for all other finite subsets $\setu$. We obviously
have that $f_{\{1,\ldots,k\}}\in F_{\{1,\ldots,k\}}$ and the sum
\[
 f_d(x_1,\ldots,x_d)
 \,:=\, \sum_{\setu\subseteq\{1,\ldots,d\}} f_\setu(\bsx)
 \,=\, \sum_{k=1}^d f_{\{1,\ldots,k\}}(\bx)
 \,=\, \psi_d(x_1)
\]
is pointwise, but not dominantly, convergent to zero for all $\bx\in
D^\bbN$. However,
\[
  \sum_{\setu\subseteq\{1,\ldots,d\}} I_\setu(f_\setu)
  \,=\, \sum_{k=1}^d I_{\{1,\ldots,k\}}\big(f_{\{1,\ldots,k\}}\big)=1
      \qquad\forall d\ge 1.
\]
This example shows that the dominated convergence of
the decomposition in \eqref{dec} is crucial.
\end{example}

As a final comment, we remark that it is often more convenient in practice
to choose the spaces $F_\setu$ to be reproducing kernel spaces based on a
simple univariate kernel, because the norms of a given function are
typically smaller. In this case the decomposition in \eqref{eq:decomp} is
uniquely determined.

\subsection{A strengthened assumption on $C_\setu$ and $B_\setu$}

For our formal setting we make a stronger assumption than
\eqref{ass:pre5}, namely we assume a certain decay
\begin{enumerate}[resume*=assumptions]
\item\label{A4} \qquad\quad %
$ \displaystyle \alpha_0 \,:=\, {\rm
  decay}(\{C_{\setu}\,B_\setu\}_\setu) \,:=\, \sup\bigg\{\alpha \;:\;
  \sum_{|\setu|<\infty} (C_{\setu}\,B_{\setu})^{1/\alpha}
   < \infty\bigg\}  \,>\, 1. $
\end{enumerate}
The purpose of this strengthened assumption will become clear in
Subsection~\ref{sec:def-calU}, when we construct the active set.

\subsection{Allowed algorithms}

In general, the components $f_\setu$ in the decomposition
$f=\sum_{|\setu|<\infty} f_{\setu}$ are not known explicitly.
Nevertheless, it is assumed that we can sample $f_\setu$ at arbitrary
points $\bsx_\setu$ in the domain. Explicitly, we make the following
assumption.
\begin{enumerate}[resume*=assumptions]
\item\label{A5} %
For a finite set $\setu\subset \bbN$ we can evaluate
$f_\setu(\bsx_\setu)$ for $\bsx_\setu\in D^{|\setu|}$ at cost
$\pounds(|\setu|)$, where $\pounds$ is a given non-decreasing
function.
\end{enumerate}
At this point we make no assumption about our ability to evaluate
$f(\bx)$, but we shall return to this question in
Section~\ref{sec:applic1}.

\medskip
We assume that for each $\setu$ we have at our disposal a sequence
$\left\{A_{\setu,n}\right\}_{n\in\bbN \cup \{0\}}$ of quadrature rules
approximating $I_\setu(f_\setu)$ as in \eqref{eq:alg}, with $A_{\setu,0}=0$,
and, moreover, the following condition is satisfied:
\begin{enumerate}[resume*=assumptions]
\item\label{A6} There exists $q>0$ with the following property:
    for each $\setu$, there exist $G_{\setu,q}>0$ such that the
    \emph{worst case error} of $A_{\setu,n}$ in the unit ball of
    $F_\setu$ satifies
\begin{equation} \label{eq:Gu}
 \|I_\setu-A_{\setu,n}\|
 \,=\,  \sup_{\|f_\setu\|_{F_\setu}\le1} |I_\setu(f_\setu)-A_{\setu,n}(f_\setu)|
 \,\le\, \frac{G_{\setu,q}}{(n+1)^q}  \qquad\mbox{for}\quad
 n\,=\,0,1,2,\dots.
\end{equation}
\end{enumerate}
Note that $q$ is not uniquely defined. Since~\eqref{eq:Gu} holds even for
$n=0$, we can assume $C_{\setu}\le G_{\setu,q}$ where $C_\setu$ is as in
\eqref{eq:Cu}. We also observe that \eqref{eq:Gu} implies
$\lim_{n\to\infty}\|I_\setu-A_{\setu,n}\| = 0$.

\section{Multivariate Decomposition Method} \label{sec:mdm}

We are now ready to introduce the \emph{Multivariate Decomposition Method}
(MDM) in our setting.

\subsection{MDM}

As in \cite{Was13a,Was14a,PW14}, the first step of the method is to
construct, for given $\e>0$, what we call here the \emph{active set}
$\calU(\e)$ -- a finite collection of those subsets
$\setu\subset\mathbb{N}$ that are most important for the integration
problem. Specifically, under our standing assumptions
\ref{A1}--\ref{A6}, we choose a set $\calU(\e)$ such that
\begin{equation} \label{eq:term1}
  \sum_{\setu\notin\setU(\e)}|I_\setu(f_\setu)|\,\le\, \frac{\e}{2}.
\end{equation}
The MDM $\calA_\e(f)$ for the integral $\calI(f)$ is then given by
\eqref{eq:alg}, with the values of $n_\setu$ chosen such that
\begin{equation}\label{eq:term2}
  \sum_{\setu\in\setU(\e)}|I_\setu(f_\setu)-A_{\setu,n_\setu}(f_\setu)|
   \,\le\, \frac{\e}{2}.
\end{equation}
It then follows that for all $f\in\calF$ the integration error of MDM
satisfies
\[
  |\calI(f)-\calA_\e(f)|
  \,\le\, \sum_{\setu\notin\setU(\e)}|I_\setu(f_\setu)|
  + \sum_{\setu\in\setU(\e)}|I_\setu(f_\setu)-A_{\setu,n_\setu}(f_\setu)|
  \,\le\, \frac{\e}{2} + \frac{\e}{2} \,=\, \e.
\]
Consequently, the \emph{worst case error} of $\calA_\e$ in $\calF$
satisfies
\[
  e(\calA_\e;\calF) \,:=\, \sup_{f\in\calF} |\calI(f)-\calA_\e(f)| \,\le\, \e.
\]
The \emph{information cost} is
\begin{equation} \label{eq:cost}
  {\rm cost}(\calA_\e)\,:=\,\sum_{\setu\in\setU(\e)}
  n_\setu\,\pounds(|\setu|).
\end{equation}

We remark that such an active set $\calU(\e)$ is not unique. Note the two
distinct special cases $\calU(\e)=\emptyset$ (which corresponds to
$\calA_\e=0$) and $\calU(\e)=\{\emptyset\}$ (which corresponds to
$\calA_\e(f)=f_\emptyset$).

\subsection{Constructing $\calU(\e)$}\label{sec:def-calU}

We use essentially the same approach for constructing the active set
$\calU(\e)$ as in \cite{Was13a}. Recall that for each $\setu$ we have
$|I_\setu(f_\setu)|\le C_{\setu} B_\setu$, and that the sequence
$\{C_\setu B_\setu\}$ satisfies \ref{A4}.  It follows that for any
$\alpha \,\in\, (1,\alpha_0)$\, we may define
\begin{equation} \label{eq:def_U}
  \calU(\varepsilon)\,=\,\calU(\e,\alpha)
  \,:=\, \bigg\{ \setu \;:\; (C_{\setu}\,B_{\setu})^{1-1/\alpha} \, > \,
  \frac{\e/2}{\sum_{|\setv|<\infty} (C_{\setv}\,B_{\setv})^{1/\alpha}}
  \bigg\},
\end{equation}
and this would yield \eqref{eq:term1}. Moreover, following the proof of
\cite[Theorem~2]{Was13a} we can obtain an upper bound on the size of the
resulting active set, as given in the proposition below.

\begin{proposition} \label{prp:1}
Let $\calU(\e,\alpha)$ be given by~\eqref{eq:def_U}.  Then for any
$\varepsilon>0$ and $\alpha\in(1,\alpha_0)$ we have
\[
   \left|\setU(\e,\alpha)\right|\,<\,\left(\frac2\e\right)^{\tfrac{1}{\alpha-1}}
     \Bigg(\sum_{|\setu|<\infty}(C_{\setu}\,B_{\setu})^{\tfrac{1}{\alpha}}
       \Bigg)^{\tfrac{\alpha}{\alpha-1}}.
\]
\end{proposition}

\subsection{Constructing $A_{\setu,n_\setu}(f_\setu)$}\label{sec:algs}

The main difficulty in the construction of the algorithms
$A_{\setu,n_\setu}$ for $\setu\in\setU(\e,\alpha)$ is the selection of the
numbers $n_\setu$. A natural approach is to minimize the information
cost~\eqref{eq:cost} subject to the desired error bound~\eqref{eq:term2}
being attained. This depends on the rate of convergence of the worst case
errors $\|I_\setu-A_{\setu,n}\|$ for fixed $\setu$ and $n\to\infty$. For a
given selection of the $n_\setu$ this rate is determined by \eqref{eq:Gu},
from which it follows that for any $f\in\calF$ we have
\begin{equation} \label{eq:need1}
 \sum_{\setu\in \calU(\varepsilon,\alpha)}
     |I_\setu(f)-A_{\setu,n_{\setu}}(f)| \,\le\,
     \sum_{\setu\in \calU(\varepsilon,\alpha)}
   \frac{G_{\setu,q}\,B_{\setu}}{(n_\setu+1)^q}.
\end{equation}

Observe that if we take
\[
  n_\setu\,=\,n_\setu(\e,q)\,=\,\left\lfloor h_\setu\right\rfloor,
\]
where the positive real numbers $h_\setu$ minimize
$\sum_{\setu\in\calU(\e,\alpha)}h_\setu\,\pounds(|\setu|)$ subject to
$\sum_{\setu\in \calU(\e,\alpha)} G_{\setu,q}\,B_{\setu}/h_\setu^q =
\e/2$, then both the error and the information cost are controlled, since
\[
  \sum_{\setu\in\setU(\e,\alpha)}\frac{G_{\setu,q}\,B_{\setu}}{(n_\setu+1)^q}
  \,\le\,\sum_{\setu\in\setU(\e,\alpha)}\frac{G_{\setu,q}\,B_{\setu}}{h_\setu^q}
  \,=\,  \frac\e2\quad\mbox{and}\quad
   \sum_{\setu\in\setU(\e,\alpha)} n_\setu\,\pounds(|\setu|)\,\le\,
   \sum_{\setu\in\setU(\e,\alpha)} h_\setu\,\pounds(|\setu|).
\]
An explicit formula for such $h_\setu$ can be obtained using a Lagrange
multiplier argument, giving
\begin{equation}\label{def:nu}
  h_\setu \,=\,
  \bigg(\frac{2}{\e}\sum_{\setv\in \calU(\e,\alpha)}
  \pounds(|\setv|)^{q/(q+1)}\,(G_{\setv,q}\,B_{\setv})^{1/(q+1)}
     \bigg)^{1/q}
  \left(\frac{G_{\setu,q}\,B_{\setu}}
      {\pounds(|\setu|)}\right)^{1/(q+1)}.
 \end{equation}
This analysis leads to the following theorem.

\begin{theorem}\label{thm:1}
Under the standing assumptions \ref{A1}--\ref{A6},
for any $\varepsilon>0$ and $\alpha\in (1,\alpha_0)$ the algorithm
$\calA_\eps$ with $\setU(\e)=\setU(\e,\alpha)$ defined by~\eqref{eq:def_U}
and $n_\setu=\lfloor h_\setu \rfloor$ with $h_\setu$ defined
by~\eqref{def:nu} produces an approximation to the integral $\calI$ with
worst case error $e(\calA_\e;\calF) \le \e$ and
\begin{align}
  \nonumber
  {\rm cost}(\calA_\e)
  \,\le \sum_{\setu\in\calU(\e,\alpha)} h_\setu\, \pounds(|\setu|)
  &\,\le\, \left(\frac{2}{\e}\right)^{1/q}
  \bigg(\sum_{\setu\in\calU(\e,\alpha)} \pounds(|\setu|)^{q/(q+1)}\,
  (G_{\setu,q}\,B_{\setu})^{1/(q+1)} \bigg)^{1+1/q}
  \\
  \label{eq:need2}
  &
  \,\le\, \left(\frac{2}{\e}\right)^{1/q}
  \; \bigg(\sum_{\setu\in\setU(\e,\alpha)}
  (G_{\setu,q}\,B_{\setu})^{1/(q+1)}\bigg)^{1+1/q}\;
      \max_{\setu\in\setU(\e,\alpha)}\pounds(|\setu|).
\end{align}
\end{theorem}

If we wish, we could also assume, analogously to Assumption \ref{A4}, that
\[
  \alpha_q \,:=\, {\rm
    decay}(\{G_{\setu,q}\,B_\setu\}_\setu) \,:=\, \sup\bigg\{\tau \ :\
    \sum_{|\setu|<\infty} (G_{\setu,q}\,B_{\setu})^{1/\tau} <
     \infty\bigg\} \,>\, 1.
\]
In that case, if it also happens that $q<\alpha_q-1$,
then the sum in~\eqref{eq:need2} would be uniformly bounded for all $\e>0$,
\[
  \sum_{\setu\in\calU(\e,\alpha)} (G_{\setu,q}\,B_{\setu})^{1/(q+1)}
  \,\le\, \sum_{|\setu|<\infty} (G_{\setu,q}\,B_{\setu})^{1/(q+1)}\,<\,\infty.
\]
Then, up to a constant, the cost in Theorem \ref{thm:1} would be upper
bounded by $\varepsilon^{-1/q}\max_{\setu\in\setU(\e,\alpha)}\pounds(|\setu|).$

\smallskip
Theorem~\ref{thm:1} differs from \cite[Theorem~9]{Was13a} in the following
way: the assumed bound on the errors for the corresponding algorithms
$A_{\setu,n_\setu}$ in \cite[Formula~(18)]{Was13a} involve also some
logarithmic factors in $n_\setu$ which results in an overall error that is
bounded by $\e$ times a small factor depending on $\e$. Later in
Subsection~\ref{ssec:app1smol} we will also encounter such a scenario, and
we will relax the requirement~\eqref{eq:term2} a little by neglecting
those multiplicative factors when choosing $n_\setu$. In this way we
obtain an algorithm $\overline{\calA}_\eps$ whose error is slightly larger
than $\varepsilon$.

\begin{remark} For a practical implementation we note that the
construction of the active set $\calU(\eps)=\calU(\eps,\alpha)$
in~\eqref{eq:def_U} depends on the error request $\e$ and the choice of
the decomposition through the spaces $F_\setu$ and hence on values of
$C_\setu$ and $B_\setu$, and additionally on the choice of the
parameter~$\alpha$. Moreover, the infinite sum from the denominator
in~\eqref{eq:def_U} needs to be estimated from above. The values of $q$
and $G_{\setu,q}$ in~\eqref{eq:Gu} depend on the available algorithms
$A_{\setu,n_\setu}$ and enter the formula for $n_\setu$
via~\eqref{def:nu}.
\end{remark}

\section{First application: anchored RKHS}\label{sec:applic1}

In both this and the next sections, the space $F_\setu$ is an anchored
space, anchored at $0$. This means that for $\setu\ne\emptyset$, every
function $f_\setu\in F_\setu$ satisfies \eqref{eq:kill} with $a=0$,
and that if $\setu\ne\setv$ then $F_\setu\cap F_\setv$ contains only
the zero function. In Subsection~\ref{sec:A5} we state an explicit formula
for the anchored decomposition to demonstrate how \ref{A5} holds.

In Subsection~\ref{sec:aRKHS}, we define the setting for the case of a
general anchored reproducing kernel and a general domain $D$. In
Subsection~\ref{ssec:ProbForm} we specialize the integration domain to
$[-\tfrac{1}{2}, \tfrac{1}{2}]$ and the kernel to one that leads to a
subspace of the space described in Subsection~\ref{ssec:example}, with a
redefined (but equivalent) norm. In Subsection~\ref{ssec:examples} we
consider again the motivating example from \eqref{eq:example}. In
Subsection~\ref{ssec:lattice} we specialize the quadrature formulas to
lattice rules. Then, in Subsection~\ref{ssec:app1smol} we specialize the
quadrature formulas to Smolyak's quadrature, and finally develop the
algorithm $\overline{\calA}_\eps$.

\subsection{Anchored decomposition and \ref{A5}} \label{sec:A5}

Recall that \ref{A5} is about the cost of evaluating individual terms
$f_\setu$ from $f=\sum_{|\setu|<\infty}f_\setu$. It is shown in
\cite{KSWW10b} that, for the anchored decomposition with anchor at $0$,
the value of $f_\setu(\bsx_\setu)$ can be expressed as a combination of at
most $2^{|\setu|}$ values of $f$, specifically
\[
  f_\setu(\bsx_\setu)\,=\,\sum_{\setv\subseteq\setu}(-1)^{|\setu|-|\setv|}
     f(\bsx_\setv;\bszero)\,,
\]
where $f(\bsx_\setv;\bszero)$ indicates that we evaluate $f(\bsx)$ with
$x_j$ set to $0$ for $j\notin\setv$. Here we assume as in, e.g.,
\cite{KSWW10a}, that we can sample $f(\bsx)$ at some cost provided that
$\bsx$ has only finitely many $x_j$ different from~$0$. More precisely, we
suppose that we can evaluate $f(\bsx_\setu;\bszero)$ at the cost
$\$(|\setu|)$, where $\$:\{0,1,2,\dots\}\to (0,\infty)$ is a given
non-decreasing cost function. It follows that $f_\setu(\bsx_\setu)$ can be
obtained at a cost of
\begin{equation}\label{eq:pound}
   \pounds(|\setu|) \,=\,
    \sum_{k=0}^{|\setu|}\binom{|\setu|}{k}\,\$(k)
   \,\le\, 2^{|\setu|}\,\$(|\setu|).
\end{equation}

\subsection{The anchored RKHS setting} \label{sec:aRKHS}

Let $F=H(K)$ be a reproducing kernel Hilbert space of univariate functions
with the kernel $K:D\times D\to\bbR$. We assume that $K$ has an anchor
$0\in D$, i.e., $K(0,0)=0$. For $g\in F$, it follows from the reproducing
property $g(x) = \il g,K(x,\cdot)\ir_F$ that
$|g(0)|\le\|g\|_F\,\|K(0,\cdot)\|_F=\|g\|_F\,\sqrt{K(0,0)}=0$, implying
$g(0)=0$ and, as a special case, $K(x,0)=0$ for all $x\in D$.

For nonempty $\setu$, the space $F_\setu$ is defined to be the reproducing
kernel Hilbert space with kernel
\begin{equation} \label{eq:Ku}
   K_\setu(\bsx_\setu,\bsy_\setu)\,=\,\prod_{j\in\setu}K(x_j,y_j).
\end{equation}
That is, $F_\setu\,=\,H(K_\setu)$ is the $|\setu|$-fold tensor product of
the space $F$ and consists of functions whose variables are those listed
in $\setu$.

In the space $F_\setu$ so defined, point evaluation is a continuous linear
functional: indeed for $\bsx_\setu \in D^\setu$ and $g_\setu\in F_\setu$
we have $g_\setu(\bsx_\setu)=\il
g_\setu,K_\setu(\bsx_\setu,\cdot)\ir_{F_\setu}$, and hence
$|g_\setu(\bsx_\setu)|\le
\|g_\setu\|_{F_\setu}\,\|K_\setu(\bsx_\setu,\cdot)\|_{F_\setu}
=\|g_\setu\|_{F_\setu} \left(K_\setu(\bsx_\setu,\bsx_\setu)\right)^{1/2},
$
from which it is easily seen that the norm of the point evaluation
functional is
\[
 \sup_{\|g_\setu\|_{F_\setu}\le 1} |g_\setu(\bsx_\setu)|
 \,=\, \left(K_\setu(\bsx_\setu,\bsx_\setu)\right)^{1/2} =
\bigg(\prod_{j\in\setu}K(x_j,x_j)\bigg)^{1/2}.
\]

We may now define $\calF$ to be the class of functions
\begin{equation} \label{eq:class}
 f(\bsx) \,=\, \sum_{|\setu|<\infty}f_\setu(\bsx_\setu)
 \quad\mbox{with}\quad f_\setu \in F_\setu
 \quad\mbox{and}\quad \|f_\setu\|_{F_\setu}\le B_\setu,
\end{equation}
and such that the above series is uniformly convergent for all
$\bsx\in D^\bbN$. The class $\calF$ can be relatively large when the
kernel $K$ is bounded. Suppose that $\|K\|_\infty\,:=\,\sup_{x\in
D}K(x,x)\,<\,\infty$. Then it follows from the reproducing property of
$K_\setu$ and from~\eqref{eq:Bu} that the terms in the decomposition
\eqref{eq:decomp} satisfy
\begin{align*}
  \sum_{|\setu|<\infty}|f_\setu(\bsx_\setu)|
  &\,=\,\sum_{|\setu|<\infty}\il f_\setu, K_\setu(\bsx_\setu,\cdot)\ir_{F_\setu}
  \,\le\,
  \sum_{|\setu|<\infty}\|f_\setu\|_{F_\setu}\,(K_\setu(\bsx_\setu,\bsx_\setu))^{1/2}
  \,\le\,  \sum_{|\setu|<\infty}B_\setu\, \|K\|_\infty^{|\setu|/2}.
\end{align*}
Hence, if
$\sum_{|\setu|<\infty}B_\setu\,\|K\|_\infty^{|\setu|/2}\,<\,\infty$ then
the convergence in \eqref{eq:decomp} is automatically uniform, and we
may define $\calF$ as the class of functions \eqref{eq:class} without
further restriction. On the other hand, if the kernel $K$ is unbounded
then $\|K\|_\infty=\infty$, and the class $\calF$ can be small. An example
of such a situation is provided by
\[
   D\,=\,\bbR \quad\mbox{and}\quad
   K(x,y)\,=\, \frac{|x|+|y|-|x-y|}{2}.
\]

\subsection{Specializing the kernel}\label{ssec:ProbForm}

We now apply our results to a special case of the reproducing kernel
Hilbert space setting.
We let
\begin{equation} \label{eq:ker}
  D \,=\,\left[-\tfrac12,\tfrac12\right],\quad
  \rho(x)=1,\quad
  K(x,y) \,=\, \frac{|x|+|y|-|x-y|}{2},
\end{equation}
and take $F=H(K)$ to be the corresponding reproducing kernel Hilbert
space. Since $K(0,0)=0$, this is clearly an anchored space with the anchor
$0$. Moreover, it can easily be verified that the corresponding norm is
given by
\[
 \|g\|_F^2 \,=\,\int_{-1/2}^{1/2}  |g'(x)|^2 \,\rd x.
\]
For nonempty $\setu$ with $|\setu|<\infty$, let $F_\setu = H(K_\setu)$
with kernel \eqref{eq:Ku}. Then the norm in the space $F_\setu$ is given
by
\begin{equation*}
 \|g_\setu\|_{F_\setu}^2
 \,=\,\int_{\left[-\frac{1}{2},\frac{1}{2}\right]^{|\setu|}}\left|
 \frac{\partial^{|\setu|}}{\partial\bsx_\setu} g_\setu(\bsx_\setu)
 \right|^2 \,\rd\bsx_\setu\quad \mbox{for any\ }g_\setu\in F_\setu,
\end{equation*}
where $\partial^{|\setu|}/\partial \bsx_\setu = \prod_{j\in\setu}
(\partial/\partial x_j)$.
For a function $f\in\calF$ with anchored decomposition $f
=\sum_{|\setu|<\infty} f_\setu$, we now have
\begin{equation}\label{eq:nrm}
    \int_{\left[-\frac{1}{2},\frac{1}{2}\right]^{|\setu|}}\left|
     \frac{\partial^{|\setu|}}{\partial\bsx_\setu}f(\bsx_\setu;\bszero)
     \right|^2 \,\rd\bsx_\setu
    \,=\,  \|f_\setu\|_{F_\setu}^2.
\end{equation}
which follows from $f(\bsx_\setu;\bszero) = \sum_{\setv\subseteq\setu}
f_\setv(\bsx_\setv)$ together with
$\partial^{|\setu|}f_\setv/\partial\bsx_\setu = 0$ if $\setv$ is a proper
subset of $\setu$.

For univariate integration $I(g) = \int_{-1/2}^{1/2} g(x)\,\rd x$
with $g\in F =H(K)$, it is easy to verify that $c_0 := \|I\|_F =
12^{-1/2}$. Due to the tensor product structure of $F_\setu$, we have
from~\eqref{eq:Cu} that
\begin{equation}\label{eq:C0ex}
  C_{\setu}\,=\,\|I_\setu\| \,=\, c_0^{|\setu|} \,=\, 12^{-|\setu|/2}.
\end{equation}

\subsection{Examples of integration problems}\label{ssec:examples}

Consider again the example in \eqref{eq:example}, but now with the
more convenient choice $F_\setu=H(K_\setu)$. For $\setu\subset\bbN$, we
then have
\begin{equation}\label{eq:trunc}
 f(\bsx_\setu;\bszero)\,=\,\frac{1}{1+\sum_{j\in\setu} x_j/j^2}
 \quad\mbox{and}\quad
  \frac{\partial^{|\setu|}}{\partial\bsx_\setu} f(\bsx_\setu;\bszero)
  \,=\,\frac{(-1)^{|\setu|}\,|\setu|!}{(\prod_{j\in\setu}j^2)
  (1+\sum_{j\in\setu} x_j/j^2)^{|\setu|+1}},
\end{equation}
and hence from~\eqref{eq:nrm}, together with $x_j\ge -1/2$
and $\sum_{j=1}^\infty 1/j^2=\pi^2/6$,
\[
   \|f_\setu\|_{F_\setu}\,\le\, \left(1-\tfrac{\pi^2}{12}\right)^{-1-|\setu|}
  \, |\setu|!\,\prod_{j\in\setu}j^{-2}\,=:\,B_\setu.
\]
When combined with $\|K\|_\infty = 1/2$ and~\eqref{eq:C0ex}, this
gives
\[
 \|K\|_\infty^{|\setu|/2}\, B_\setu =
 2^{-|\setu|/2}\left(1-\tfrac{\pi^2}{12}\right)^{-1-|\setu|}\,
 |\setu|!\,\prod_{j\in\setu}j^{-2}
 \quad\mbox{and}\quad
 C_{\setu}\,B_\setu\,=\, 12^{-|\setu|/2}
   \left(1-\tfrac{\pi^2}{12}\right)^{-1-|\setu|}\,
   |\setu|!\,\prod_{j\in\setu}j^{-2}.
\]
It follows from the lemma below that
$\sum_{|\setu|<\infty}\|K\|_\infty^{|\setu|/2} B_\setu <\infty$ (and hence
the convergence in \eqref{eq:decomp} is uniform), and that $\alpha_0$ in
\ref{A4} is given by $\alpha_0\,=\,{\rm
decay}(\{C_{\setu}\,B_{\setu}\}_\setu) = 2$.

\begin{lemma}\label{lem:u}
Let $b_1\ge 0$. Suppose the sequence $\{g_j\}_{j\ge 1}$ with $g_j>0$ has
\begin{equation*}
{\rm decay}(\{g_j\}_{j\ge 1})\,:=\, \sup\bigg\{\tau\ :\
   \sum_{j=1}^\infty g_j^{1/\tau}<\infty\bigg\} \,=:\, b_2 \,
    >\, \max(b_1,0).
\end{equation*}
Then
\[
{\rm decay}\bigg(\bigg\{(|\setu|!)^{b_1}\prod_{j\in\setu}g_j\bigg\}_\setu
    \bigg)\,=\,b_2.
\]
\end{lemma}
\begin{proof}
This is a special case of \cite[Theorem~5]{DicGne14b}. Here we provide a
simple proof. For any $\tau\in(b_1,b_2)$,
\[
\sum_{|\setu|<\infty}(|\setu|!)^{b_1/\tau}\prod_{j\in\setu} g_j^{1/\tau}
\,=\, \sum_{\ell=0}^\infty
  (\ell!)^{b_1/\tau}\sum_{|\setu|=\ell}\prod_{j\in\setu}g_j^{1/\tau}
   \le \sum_{\ell=0}^\infty (\ell!)^{b_1/\tau-1}
    \bigg(\sum_{j=1}^\infty g_j^{1/\tau}\bigg)^\ell\,<\,\infty,
\]
where the first inequality follows from $(\sum_{j=1}^\infty a_j)^\ell
\ge\ell!\sum_{|\setu|=\ell}\prod_{j\in\setu}a_j$, and in the last
expression the finiteness of the sum over $j$ follows from $\tau <b_2$,
and the finiteness of the sum over $\ell$ follows from the ratio test
using $\tau>b_1$.
\end{proof}

This example motivates us to consider in the rest of this section the case
in which \ref{A2} holds with
\begin{equation}\label{eq:POD}
  B_\setu\,=\, (|\setu|!)^{b_1} \, \mu\,  \prod_{j\in\setu}
  (\kappa\, j)^{-b_2}
  \quad\mbox{for some}\quad b_2 > \max(b_1,0)
\quad\mbox{and some}\quad \mu,\kappa>0.
\end{equation}
The lemma with~\eqref{eq:C0ex} then gives
\[
  \alpha_0 \,=\, {\rm decay}(\{C_{\setu}\,B_{\setu}\}_\setu) \,=\, b_2.
\]
It is easy to verify the following proposition, see \cite{PW11}.

\begin{proposition}\label{prp:2}
For $C_{\setu}=12^{-|\setu|/2}$ and $B_\setu$ satisfying~\eqref{eq:POD},
for $\calU(\e,\alpha)$ defined by~\eqref{eq:def_U} with $\e>0$ and
$\alpha\in (1,b_2)$, we have
\[  d(\e)\,:=\,\max_{\setu\in\setU(\e,\alpha)}|\setu|\,=\,
    O\left(\frac{\ln(1/\e)}{\ln(\ln(1/\e))}\right)
    \qquad\mbox{as}\quad \e\to 0.
\]
If $\pounds(d)=\ex^{O(d)}$ as $d\to\infty$, then the cost of evaluating
$f_\setu(\bsx_\setu)$ for $\setu\in\calU(\eps,\alpha)$ is
\[
   \pounds(d(\eps))\,=\,(1/\eps)^{O(1/\ln(\ln(1/\e)))}
   \qquad\mbox{as}\quad \e\to0.
\]
\end{proposition}

We end this subsection with the following two remarks.

\begin{remark}\label{rem:1}
Suppose that evaluation of $f([\bsx;\setu])$ incurs exponentially large
cost $\$(|\setu|)=\ex^{O(|\setu|)}$. Then it follows from~\eqref{eq:pound}
that the cost $\pounds(|\setu|)$ of obtaining $f_\setu(\bsx_\setu)$ for
$u\in\calU(\eps,\alpha)$ is still of order $\ex^{O(|\setu|)}$, and hence,
by Proposition~\ref{prp:2}, the cost is only
\[
  (1/\e)^{O(1/\ln(\ln(1/\e)))}.
\]
\end{remark}

\begin{remark}\label{rem:2}
Although Proposition~\ref{prp:2} limits very efficiently the cardinality
of the largest subset in the active set, Proposition~\ref{prp:1} suggests
that the cardinality of the active set itself is still polynomial in
$1/\e$. In particular, the active set may contain $\{1\},\dots,\{j\}$ for
a large value of~$j$. For the example in~\eqref{eq:example} we can use the
following argument to limit the size of the largest label $j$ that needs
to be considered. This is achieved by estimating the truncation error more
accurately for the specific example, rather than by applying bounds on the
worst case error.

For this example, it follows from~\eqref{eq:example} and~\eqref{eq:trunc}
together with Taylor's theorem \textnormal{(}expanding the univariate
function $1/(a+y)$ about $y=0$, with $a = 1+\sum_{j\in\setu} x_j/j^2$ and
$y = \sum_{j\notin\setu}x_j/j^2$\textnormal{)} that
\[
    f(\bsx)-f(\bsx_\setu;\bszero)\, =\,-\frac{1}{(1+\sum_{j\in\setu}
   \frac{x_j}{j^2})^2}\sum_{j\notin\setu}\frac{x_j}{j^2}
   + \frac{1}{(1+\zeta(\bsx,\setu))^3}
   \bigg(\sum_{j\notin\setu}\frac{x_j}{j^2}\bigg)^2,
\]
for some $\zeta(\bsx,\setu) \in (-\frac{\pi^2}{12},\frac{\pi^2}{12})$.
Since the integral of the first term vanishes, we have
\[
  \calI[f(\cdot)-f(\cdot_\setu;\bszero)]
  \,\le\, \frac{1}{(1-\frac{\pi^2}{12})^3}\,
  \int_{[-\tfrac{1}{2},\tfrac{1}{2}]^{\bbN}}
  \bigg(\sum_{j\notin\setu}\frac{x_j}{j^2}\bigg)^2 \,\rd \bsx
  \,=\,
  \frac{1}{12\,(1-\frac{\pi^2}{12})^3}\,\sum_{j\notin \setu}\frac{1}{j^4}.
\]
In particular, if we choose $\setu=\{1:\ell\}:=\{1,2,\ldots,\ell\}$ then
we have
\[
 \calI[f(\cdot)-f(\cdot_{\{1:\ell\}};\bszero)]\,\le\,
 \frac1{36\,(1-\frac{\pi^2}{12})^3}\,\ell^{-3}.
\]
We now take $\ell=\ell_\e$ so that the right hand side is less than
$\e/3$, implying $\ell=\ell_\e=\Omega(\e^{-1/3})$. Then we replace the sum
$\sum_{|\setv|<\infty} (C_\setv\, B_\setv)^{1/\alpha}$ in \eqref{eq:def_U}
by $\sum_{\setv\in \{1:\ell\}} (C_\setv\, B_\setv)^{1/\alpha}$, replace
\eqref{eq:term1} by
\[
  \sum_{\setu\in \{1:\ell\} \setminus\calU(\varepsilon)} |I_\setu (f_\setu)| \le \frac{\e}{3},
\]
and replace $\e/2$ in~\eqref{eq:term2} by $\e/3$. Then MDM can be run as
usual, and will still give an error bounded by $\e$, but with the
simplification that subsets containing numbers bigger than $\ell_\e$ need
never be considered.  In effect, for this example the problem can be
considered as an $\ell_\e$-dimensional problem, rather than as an
infinite-dimensional one.
\end{remark}

\subsection{Specializing the quadrature to lattice rules}\label{ssec:lattice}

It can be shown by an adaptation of known results (see, e.g.,
\cite[Theorem~5.9]{DKS13}) that in the case of the kernel in
Subsection~\ref{ssec:ProbForm} we can construct \emph{shifted lattice
rules} with $n$ points in $|\setu|$ dimensions, where $n\ge 3$ is prime,
such that \eqref{eq:Gu} holds for all $q\in [1/2,1)$,
with
\[
  G_{\setu,q} \,=\, 2^q \left(\frac{2\zeta(1/q)}{(2\pi^2)^{1/(2q)}}
  +12^{-1/(2q)}\right)^{|\setu|q},
\]
where $\zeta(x) = \sum_{k=1}^\infty k^{-x}$ is the Riemann zeta function.
More precisely, the above result is adapted from known results for lattice
rules by setting the weight $\gamma_\setu$ for the particular $\setu$ to
be $1$ and all other weights to be $0$. We follow the analysis of
Subsection~\ref{sec:algs}, but instead of taking $n_\setu = \lfloor
h_\setu \rfloor$ we take $n_\setu$ to be the largest prime number such
that $3\le n_\setu\le h_\setu$, or set $n_\setu= 0$ if this is not
achievable. The remaining analysis in that subsection then applies.

More precisely, a shifted lattice rule with $n$ points for a $d$-variate
function $g$ defined over $[-1/2,1/2]^d$ takes the form
\[
  \frac{1}{n} \sum_{i=1}^n g\left(\left\{\frac{i\bsz}{n} +
  \bsDelta\right\} - \boldsymbol{\tfrac{1}{2}}\right),
\]
where $\bsz\in \bbZ^d$ is the \emph{generating vector} and $\bsDelta\in
[0,1]^d$ is the \emph{shift}. The braces around a vector indicate that we
take the fractional part of each component in the vector, and the
subtraction by $1/2$ from all components takes care of the translation
from the standard unit cube $[0,1]^d$ to $[-1/2,1/2]^d$. A good generating
vector for the lattice rule can be constructed using the fast
component-by-component algorithm, see e.g., \cite{Nuy2014}. The shift can
be generated randomly from the uniform distribution on $[0,1]^d$ (in this
case the error bound holds in the root mean square sense), or the shift
can be generated repeatedly until the desired error bound is achieved (in
this case the error bound holds deterministically but the result is not
fully constructive). See, e.g., \cite{DKS13} for details.

\subsection{Specializing the quadrature to Smolyak's method}
            \label{ssec:app1smol}

We now apply Smolyak's \cite{Smol63} quadrature scheme to the MDM in the
RKHS context of this section, with kernel~\eqref{eq:ker}. Smolyak's
construction is often used for tensor-product problems.  It is built from
a single family of univariate quadrature rules, and for every space
$F_\setu$ of a given dimensionality $d = |\setu|$ we use the same family
of rules.
In the following we take the univariate quadrature
rules to be trapezoidal rules since in this setting they achieve the
optimal convergence rate of order $1$.

For $d\ge 1$, Smolyak's construction for approximating a $d$-variate
integral is given by the formula
\[
  Q_{d,\kappa}
  \,=\, \sum_{\bsi\in \bbN^d,\,|\bsi|\le\kappa}
  \bigotimes_{j=1}^d (U_{i_j} - U_{i_j-1}),
\]
where $|\bsi| = i_1+i_2 + \cdots +i_d$, $U_0$ is the zero algorithm, and
each $U_i$ for $i\ge 1$ is a (composite) trapezoidal rule with $2^i+1$
equally spaced points $t_{i,k} = -1/2 + k/2^i$, $k=0,\dots,2^i$.
Actually, we only need $2^i$ evaluations for $U_i$ since for every $g\in
F$ the value $g(0)=0$ is for free. Note that $Q_{1,\kappa}=U_\kappa$. In
general, if $\kappa < d$ then $Q_{d,\kappa}\equiv0$.

It is easy to verify that for univariate integration we have
\[
  \int_{-1/2}^{1/2}g(t)\,\rd t-U_i(g)\,=\,\int_{-1/2}^{1/2}
   g'(t)\,K_i(t)\,\rd t
  \, \le\, \|g\|_F\,\bigg( \int_{-1/2}^{1/2} K_i^2(t)\,\rd t\bigg)^{1/2},
\]
with $K_i(t) = (t_{i,k}+t_{i,k+1})/2-t$ if $t\in[t_{i,k},t_{i,k+1})$.
Hence the worst case error of $U_i$ is the $L_2$ norm of $K_i$, which is
\[
  \|I-U_i\|\,=\,
   \frac1{\sqrt{12}} \, 2^{-i}\qquad\mbox{for $i=0,1,\dots$}.
\]

{}From \cite[Lemma\,1]{WW95} we know that $Q_{d,\kappa}$ can be written in
an equivalent form as
\begin{equation}\label{smolform}
  Q_{d,\kappa}\,=\,\sum_{\bsi\in P(d,\kappa)}(-1)^{\kappa-|\bsi|}\,
  \binom{d-1}{\kappa-|\bsi|}\, \bigotimes_{\ell=1}^dU_{i_\ell},
\end{equation}
where $P(d,\kappa) =
      \left\{\bsi\in\bbN^d\ :\ \kappa-d+1\le|\bsi|\le\kappa\right\}$.
Note that this holds for general building blocks $U_i$. {}From \cite[Lemma
6]{WW95} we then have the following proposition.

\begin{proposition}\label{prp:err}
For $d,\kappa\in\bbN$ with $\kappa\ge d$, the error of the
trapezoidal-Smolyak algorithm is
\[
 \|I_{\{1:d\}}-Q_{d,\kappa}\|
\,\le\,2^{-\kappa-1}\,3^{-d/2}\,
  \sqrt{\binom{\kappa}{d-1}}
  \,\le\,2^{-\kappa-1}\,3^{-d/2}\,\sqrt{\frac{\kappa^{d-1}}{(d-1)!}}.
\]
\end{proposition}

The following proposition provides bounds on the number $n(d,\kappa)$ of
function evaluations used by the trapezoidal-Smolyak algorithm
$Q_{d,\kappa}$.

\begin{proposition}\label{prp:crd}
For $d,\kappa\in\bbN$ with $\kappa\ge d$ we have $n(1,\kappa) =
2^{\kappa}$ and
\[
    2^{\kappa-d+1} \,\le\,
    n(d,\kappa)
    \,\le\,2^{\kappa-d+1}\,\mathrm{e}^{d/2-1}\,
     \frac{\kappa^{d-1}}{(d-1)!}
    \qquad\mbox{for}\quad d\ge 2.
\]
\end{proposition}

\begin{proof}
Clearly $n(1,\kappa) = 2^{\kappa}$. Let $d\ge 2$. The lower bound on
$n(d,\kappa)$ is trivial. To obtain the upper bound, we count only those
points used by $Q_{d,\kappa}$ that correspond to $\bsi=(i_1,\ldots,i_d)\in
P(d,\kappa)$ in~\eqref{smolform} with $|\bsi|=\kappa$, but do not count
those points with any component equal to $0$. The number of points
corresponding to such $\bsi$ with $i_d=1$ is $2\,n(d-1,\kappa-1)$. For
successive $s=2,3,\ldots,\kappa-d+1$, the number of points corresponding
to $\bsi$ with $i_d=s$ that have not been counted yet is
$(2^s-2^{s-1})\,n(d-1,\kappa-s)$. This yields
\[
  n(d,\kappa) \,=\, 2\,n(d-1,\kappa-1)+
  \sum_{s=2}^{\kappa-d+1}2^{s-1}\,n(d-1,\kappa-s),
  \qquad d\ge 2.
\]
We now define
\[
  b(d,\kappa) \,:=\, \frac{n(d,\kappa)}{2^{\kappa-d+1}},\qquad\kappa\ge d\ge 1.
\]
Then we have $b(1,\kappa)=1$ and
\begin{equation}\label{ndk}
  b(d,\kappa)\,=\,b(d-1,\kappa-1)+\sum_{s=d-1}^{\kappa-1}b(d-1,s),\qquad d\ge 2.
\end{equation}

It remains to show that for $d\ge 2$
\begin{equation}\label{bebe}
  b(d,\kappa)\,\le\,\frac{\mathrm{e}^{d/2-1}\,\kappa^{d-1}}{(d-1)!}.
\end{equation}
Since $b(2,\kappa)=\kappa$ owing to~\eqref{ndk}, inequality~\eqref{bebe}
holds for $d=2$. Suppose~\eqref{bebe} holds for some $d\ge2$. Using
$b(d,\kappa-1)\le b(d,\kappa)$ and the induction hypothesis, we obtain
from~\eqref{ndk} that
$$
   b(d+1,\kappa)\,=\,b(d,\kappa-1)+\sum_{s=d}^{\kappa-1}b(d,s)
   \,\le\,\sum_{s=d}^\kappa b(d,s)
   \,\le\, \frac{\mathrm{e}^{d/2-1}}{(d-1)!}\sum_{s=d}^\kappa s^{d-1}
   \,\le\, \frac{\mathrm{e}^{(d+1)/2-1}\,\kappa^d}{d!},
$$
where we used the fact that $x\mapsto x^{d-1}$ is a convex function so
that
$$
  \sum_{s=d}^\kappa s^{d-1}\,\le\,\int_{d-1/2}^{\kappa+1/2}x^{d-1}\,\rd x
    \,\le\,\frac{(\kappa+1/2)^{d}}{d}
    \,\le\,\frac{\mathrm e^{1/2}\kappa^{d}}{d},
$$
and in the last inequality we used $1+y\le e^y$. Thus~\eqref{bebe} follows
by induction.
\end{proof}

We now turn to the construction of the algorithm
\begin{equation}\label{Abar}
  \overline{\calA}_\eps(f)\,=\,\sum_{\setu\in\setU(\e,\alpha)} A_{\setu,n_\setu}(f_\setu).
\end{equation}
(The ``bar'' in our notation for the algorithm $\overline{\calA}_\eps$
indicates that its error is slightly larger than~$\e$, as we show in
Theorem~\ref{thm:smol} below.) Recall that the active set
$\setU(\e,\alpha)$ is given by~\eqref{eq:def_U}. For the constant term
$f_\emptyset =f(0,0,\ldots)$, we define the corresponding algorithm to be
the one-point rule
\begin{equation}\label{eq:A0}
  A_{\emptyset,n_\emptyset} (f_\emptyset) \,:=\, f_\emptyset \,=\, f(0,0,\ldots),
  \qquad\mbox{with}\quad n_\emptyset \,:=\, 1.
\end{equation}
For each nonempty $\setu\in\setU(\e,\alpha)$, we recall that $F_\setu$ is
equivalent to $H(K_d)$ with $d=|\setu|$ after an appropriate relabeling of
the variables. Therefore we define
\begin{equation} \label{eq:Au}
  A_{\setu,n_\setu} \,:=\, Q_{|\setu|,\kappa_\setu}
\end{equation}
for some $\kappa_\setu$ to be specified below. If $\kappa_\setu < |\setu|$
then $Q_{|\setu|,\kappa_\setu}$ is the zero algorithm and $n_\setu =0$;
otherwise $n_\setu= n(|\setu|,\kappa_\setu)$ is the number of function
evaluations used by the trapezoidal-Smolyak algorithm
$Q_{|\setu|,\kappa_\setu}$, see Proposition~\ref{prp:crd}.

We now express the error $\|I_\setu-A_{\setu,n_\setu}\|$ in the form
\eqref{eq:Gu},
with $q\le 1$ as is appropriate for the trapezoidal-Smolyak algorithm in
this setting. Clearly the worst case error for $\setu=\emptyset$ is zero
and so $G_{\emptyset,q} = 0$. For any nonempty
$\setu\in\setU(\e,\alpha)$ with $\kappa_\setu\ge|\setu|$, we can use
Propositions~\ref{prp:err} and~\ref{prp:crd} to obtain an upper bound on
$\|I_\setu-A_{\setu,n_\setu}\|\,(n_\setu + 1)^q$, namely
\[
  G_{\setu,q}\,=\,
  2^{-|\setu|-1+(\kappa_\setu-|\setu|+1)q}\,3^{-|\setu|/2}\,
   \mathrm{e}^{(|\setu|/2-1)q}\,
  \bigg(\frac{\kappa_\setu^{|\setu|-1}}{(|\setu|-1)!}\bigg)^{1/2+q}.
\]
When $\kappa_\setu < |\setu|$ and so $n_\setu= 0$, the above error bound
holds with $G_{\setu,q}=12^{-|\setu|/2}$.

For each nonempty $\setu\in\setU(\e,\alpha)$, let $h_\setu$ be given
by~\eqref{def:nu} with $q\le 1$ and $G_{\setu,q} = 1$, and define
\begin{equation}\label{def:ku}
 \kappa_\setu\,:=\,|\setu|+\lfloor\log_2 h_\setu\rfloor,
\end{equation}
so that
\begin{equation} \label{eq:km}
 2^{\kappa_\setu-|\setu|} \,\le\, h_\setu \,<\, 2^{\kappa_\setu-|\setu|+1}.
\end{equation}
Then for $\kappa_\setu\ge |\setu|$ we have from Proposition~\ref{prp:crd}
together with~\eqref{eq:km} that
\begin{equation} \label{eq:nm}
  h_\setu \,<\, n_\setu
  \,\le\,2\,h_\setu\,
  \mathrm{e}^{|\setu|/2-1}\,\frac{\kappa_\setu^{|\setu|-1}}{(|\setu|-1)!}.
\end{equation}
(Note that $n_\setu$, the number of function evaluations used by
$A_{\setu,n_\setu}$, has the same meaning here as in
Subsection~\ref{sec:algs}, but its connection with $h_\setu$ here is
different from that in Subsection~\ref{sec:algs}.)
Following~\eqref{eq:need1} with $G_{\setu,q}=1$ and using the lower bound
from~\eqref{eq:nm}, we obtain
\begin{equation} \label{eq:need11}
 \sum_{\setu\in \calU(\varepsilon,\alpha)}
     |I_\setu(f)-A_{\setu,n_{\setu}}(f)|
  \,\le\,
   \bigg( \max_{\setu\in\calU(\varepsilon,\alpha)}G_{\setu,q} \bigg)
   \bigg( \sum_{\setu\in \calU(\varepsilon,\alpha)}
     \frac{B_{\setu}}{h_\setu^q} \bigg).
\end{equation}
The upper bound from~\eqref{eq:nm} yields
\begin{align} \label{eq:need22}
  \sum_{\setu\in\calU(\e,\alpha)} n_\setu\,\pounds(|\setu|)
  &\,\le\,
  \bigg( \max_{\setu\in\calU(\varepsilon,\alpha)}
  2\,\mathrm{e}^{|\setu|/2-1}\,\frac{\kappa_\setu^{|\setu|-1}}{(|\setu|-1)!} \bigg)
  \bigg( \sum_{\setu\in\calU(\e,\alpha)} h_\setu\,\pounds(|\setu|)\bigg).
\end{align}
{}From the derivation which leads to the definition of $h_\setu$
in~\eqref{def:nu}, we conclude that the second factor on the right-hand side
of~\eqref{eq:need11} is $\varepsilon/2$, while the second factor on the
right-hand side of~\eqref{eq:need22} can be bounded as
in~\eqref{eq:need2}. This leads to the following theorem.

\begin{theorem}\label{thm:smol}
For the reproducing kernel Hilbert space setting specified
by~\eqref{eq:ker}, for any $\varepsilon>0$, $\alpha\in (1,\alpha_0)$ and
$q\le 1$, the algorithm $\overline{\calA}_\eps$ with
$\setU(\e)=\setU(\e,\alpha)$ defined by~\eqref{eq:def_U},
$A_{\setu,n_\setu}$ defined by~\eqref{eq:A0} and~\eqref{eq:Au}, and
$\kappa_\setu$ defined by~\eqref{def:ku}, produces an approximation to the
integral $\calI$ with error
\[
  e(\overline{\calA}_\eps;\calF)\,\le\,\e\, X(\e,\alpha,q),
\]
and cost
\[ {\rm cost}(\overline{\calA}_\eps)\,\le\, \left(\frac{2}{\e}\right)^{1/q}
     \;\bigg(\sum_{\setu\in\setU(\e,\alpha)}
      B_{\setu}^{1/(q+1)}\bigg)^{1+1/q}\;
      \max_{\setu\in\setU(\e,\alpha)}\pounds(|\setu|)\;Y(\e,\alpha),
\]
where
\[
  X(\e,\alpha,q)\,=\, \max_{\setu\in\setU(\e,\alpha)}
   2^{-|\setu|-1+(\kappa_\setu-|\setu|+1)q}\,3^{-|\setu|/2}\,
   \mathrm{e}^{(|\setu|/2-1)q}\,
  \bigg(\frac{\kappa_\setu^{|\setu|-1}}{(|\setu|-1)!}\bigg)^{1/2+q},
\]
and
\[
  Y(\e,\alpha)\,=\, \max_{\setu\in\setU(\e,\alpha)}
     2\,\mathrm{e}^{|\setu|/2-1}\,\frac{\kappa_\setu^{|\setu|-1}}{(|\setu|-1)!}.
\]
\end{theorem}

Since $B_\setu$ of the form~\eqref{eq:POD} implies $|\setu|\le
d(\e)=O(\ln(1/\e)/\ln(\ln(1/\e)))$ for $\setu\in\setU(\e,\alpha)$, one can
show using proof techniques similar to those from \cite{PW11} that both
$X(\e,\alpha,q)$ and $Y(\e,\alpha)$ equal
\[  (1/\e)^{O(1/\ln(\ln(1/\e)))}.
\]

\begin{corollary}\label{col:firstappl}
Under the conditions of Theorem~\ref{thm:smol}, for $B_\setu$ of the
form~\eqref{eq:POD} we have
\[
   e(\overline{\calA}_\eps;\calF)\,\le\,\e^{1-\delta(\varepsilon)},
\]
where $\delta(\varepsilon) = O(1/\ln(\ln(1/\e)))$. Moreover, if
$\pounds(d)=\ex^{O(d)}$ then
\[ {\rm cost}(\overline{\calA}_\eps)\,\le\,
    2^{1+1/q}\left(\frac{1}{\e}\right)^{1/q+\delta(\varepsilon)}
     \;\bigg(\sum_{\setu\in\setU(\e,\alpha)}
      B_{\setu}^{1/(q+1)}\bigg)^{1+1/q}.
\]
\end{corollary}

\section{Second application: a non-Hilbert setting}
\label{sec:SecondAppl}

Next we consider an example which is in an anchored space setting, but not
in a Hilbert space setting.

\subsection{Problem formulation}

Let $D=\bbR_+=[0,\infty)$ and let $F$ be the space of (locally) absolutely
continuous functions $g:D\to\bbR$ such that
\[
  g(0)\,=\,0\quad\mbox{and}\quad \|g\|_F\,:=\,\|g'\|_{\infty}\,<\,\infty.
\]
The space $F_\setu$ is the completion of the $|\setu|$-fold algebraic
tensor product of $F$ whose functions depend only on variables listed in
$\setu$. The completion is with respect to
\[
  \|f_\setu\|_{F_\setu}\,:=\,
   \bigg\|\frac{\partial^{|\setu|}}{\partial\bsx_\setu}
   f_\setu\bigg\|_{\infty}.
\]
Note that $F_\setu$ is not a Hilbert space; however it is anchored at  $0$.

In the univariate case, for $g\in F$ and $x\in \bbR^+$ we can write
$g(x)=\int_0^x g^{'}(t)\,\rd t$ and therefore $|g(x)|\le
x\,\|g'\|_\infty$, from which it follows easily that the functional for
evaluation at the point $x$ has the norm $x$.  In a similar way it follows
that the point evaluation functional for the finite subset $\setu$ has the
norm $\sup_{\|g_\setu\|_{F_\setu}\le 1} |g_\setu(\bsx_\setu)| =
\prod_{j\in\setu} x_j$.

We are interested in approximating the weighted integral of $f\in\calF$,
where the weights are $\rho(x)=\exp(-x)$ for the univariate case, and
\[
  \rho_\setu(\bsx_\setu)\,:=\,\prod_{j\in\setu} \exp{(-x_j)}\,=\,
  \exp\bigg(\!-\sum_{j\in\setu} x_j\bigg)
\]
for the multivariate case. The class $\calF$ can then be defined as the
set of all uniformly
convergent sums of functions $f_\setu\in F_\setu$. To
obtain the functional for integration, note first that for the univariate
case, by integration by parts,
\begin{equation}\label{eq:Ig}
 I(g)\,:=\,\int_0^\infty f(x)\exp(-x)\,\rd x
 \,=\, \int_0^\infty f^{'}(t) \exp(-t)\,\rd t,
\end{equation}
hence the integration functional has the norm $1$, leading to $C_\setu
=\|I_\setu\|=1$.

\subsection{Smolyak's construction}

We approximate the univariate integral \eqref{eq:Ig} by algorithms $U_i$
that are weighted versions of the (composite) trapezoidal rules using the
points
\[
  x_{i,k}\,:=\, -2\ln\left(1-\frac{k}{2^i+1}\right),\qquad 0\le k\le
    2^i.
\]
Specifically, $U_0=0$, and for $i\ge 1$ we have
$U_i(f)\,=\,\sum_{k=1}^{2^i} a_{i,k}\, g(x_{i,k})$ with
\[
   a_{i,k}\,=\, \frac{\ex^{-x_{i,k+1}}-\ex^{-x_{i,k}}}{x_{i,k+1}-x_{i,k}}
       -\frac{\ex^{-x_{i,k}}-\ex^{-x_{i,k-1}}}{x_{i,k}-x_{i,k-1}},
   \quad 1\le k\le 2^i-1,
\quad\mbox{and}\quad
  a_{i,2^i}\,=\,
  -\frac{\ex^{-x_{i,2^i}}-\ex^{-x_{i,2^i-1}}}{x_{i,2^i}-x_{i,2^i-1}}.
\]

It was shown in \cite{PW14} that
\[
  \|I-U_i\|<C_1\,2^{-i}\quad\mbox{with}\quad C_1=1.00656,
  \qquad\mbox{and}\qquad\|U_i-U_{i-1}\|<2^{1-i}.
\]
This perfectly fits the setting of \cite[Lemmas 2 and 7]{WW95}.
It follows that for the corresponding Smolyak's algorithm
$A_{\setu,n_\setu}=Q_{|\setu|,\kappa_\setu}$ for
$|\setu|$-variate integrals as in~\eqref{eq:Au}, $\kappa_\setu\ge|\setu|$,
we have
\[
 \|I_\setu-A_{\setu,n_\setu}\|
 \,\le\,C_1\,2^{-(\kappa_\setu-|\setu|+1)}
    \binom{\kappa_\setu}{|\setu|-1}
\]
and
\[
  n_\setu\,=\,n(|\setu|,\kappa_\setu)\,=\,
   2^{\kappa_u-|\setu|+1}\binom{\kappa_\setu-1}{|\setu|-1}\le
   2^{\kappa_u-|\setu|+1}\binom{\kappa_\setu}{|\setu|-1}-1.
\]
Hence~\eqref{eq:Gu} holds for $q\le 1$ with
\[
  G_{\setu,q} \,=\,C_1\,2^{(q-1)(\kappa_\setu-|\setu|+1)}
   \binom{\kappa_\setu}{|\setu|-1}^{1+q}\quad\mbox{if}\quad
   \kappa_\setu\ge|\setu|,
\]
and $G_{\setu,q}=1$ if $\kappa_\setu<|\setu|$ (in which case
$n_\setu=0$ and $A_{\setu,0}=0$).

\subsection{Specializing MDM}

Taking $C_{\setu}=1$ in \ref{A4} and~\eqref{eq:def_U}, and proceeding
as in Subsection~\ref{ssec:app1smol} we obtain a result corresponding to
Corollary~\ref{col:firstappl}.

\begin{corollary}
In the setting of this section, we use the algorithm
$\overline{\calA}_\eps$ defined by~\eqref{Abar} with $q\le 1$ and
$h_\setu$ and $\kappa_\setu$ defined by~\eqref{def:nu} and~\eqref{def:ku}.
For $B_\setu$ of the form~\eqref{eq:POD} with $b_2>\max(b_1,1)$ we have
\[
  e(\overline{\calA}_\eps;\calF)\,\le\,\e^{1-\delta(\varepsilon)},
\]
where $\delta(\varepsilon) = O(1/\ln(\ln(1/\e)))$. Moreover, if
$\pounds(d)=\ex^{O(d)}$ then
\[
   {\rm cost}(\overline{\calA}_\eps)\,\le\,O\left(\e^{-(1+\delta(\varepsilon))}\right).
\]
\end{corollary}

\bigskip

\textbf {Acknowledgements}  
The research of the first and fourth authors was supported by the Australian Research 
Council under projects DP110100442, FT130100655, and DP150101770. The second 
author was partially supported by the Research Foundation Flanders (FWO) and 
the KU Leuven research fund OT:3E130287 and C3:3E150478. The research of 
the third author was supported by the National Science Centre, Poland, based on 
the decision DEC-2013/09/B/ST1/04275.  

The authors gratefully acknowledge 
the support of the Institute for Computational and Experimental Research 
in Mathematics (ICERM), as well as helfpul discussions with Michael Gnewuch, 
Mario Hefter, Aicke Hinrichs, Klaus Ritter and Henryk Wo\'zniakowski
during the preparation of this manuscript.

\vskip 3pc

\newpage
{\sc Authors addresses:}

\begin{itemize}

\item[]
{\sc Frances Y. Kuo} \\
School of Mathematics and Statistics \\
The University of New South Wales \\
Sydney NSW 2052, Australia \\
email: f.kuo@unsw.edu.au

\item[]
{\sc Dirk Nuyens} \\
Department of Computer Science, KU Leuven \\
Celestijnenlaan 200A, 3001 Leuven, Belgium \\
email: dirk.nuyens@cs.kuleuven.be

\item[]
{\sc Leszek Plaskota} \\
Faculty of Mathematics, Informatics, and Mechanics, University of Warsaw \\
ul. Banacha 2, 02-097 Warsaw, Poland \\
email: leszekp@mimuw.edu.pl

\item[]
{\sc Ian H. Sloan} \\
School of Mathematics and Statistics \\
The University of New South Wales \\
Sydney NSW 2052, Australia \\
email: i.sloan@unsw.edu.au

\item[]
{\sc Grzegorz W. Wasilkowski} \\
Department of Computer Science, University of Kentucky \\
Davis Marksbury Building, 320 Rose St., Lexington, KY 40506, USA \\
email: greg@cs.uky.edu

\end{itemize}

\end{document}